\documentclass{article}
\usepackage{amssymb,amsmath,amsfonts,amsthm}
\usepackage{latexsym}
\usepackage{mathrsfs}
\usepackage{graphics}
\usepackage{tikz}
\usetikzlibrary{shapes,arrows,calc}
\usepackage{indentfirst}

\usepackage{hyperref}

\usepackage{enumerate}
\usepackage{mathtools}
\usepackage{comment}
\usepackage{xparse}
\usepackage{bm}
\usepackage{subcaption}
\usepackage{caption}
\usepackage{graphicx} 
\usetikzlibrary{positioning} 
\hypersetup{colorlinks = true, linkcolor = blue, citecolor = blue, urlcolor = blue}
\usepackage{float}
\usepackage[section]{placeins}

\allowdisplaybreaks

\newtheorem{innercustomthm}{Theorem}
\newenvironment{customthm}[1]
  {\renewcommand\theinnercustomthm{#1}\innercustomthm}
  {\endinnercustomthm}

\newtheorem*{thm*}{Theorem}
\newtheorem{thm}{Theorem}
\newtheorem{lem}[thm]{Lemma}
\newtheorem{pro}[thm]{Proposition}
\newtheorem{obs}[thm]{Observation}

\newtheorem{ques}[thm]{Question}
\newtheorem{rem}[thm]{Remark}

\newcommand{\N}{\mathbb{N}}
\newcommand{\Z}{\mathbb{Z}}

\newcommand{\R}{\mathbb{R}}

\usepackage[margin=1in]{geometry}
\usepackage{amsmath,amssymb}
\usepackage{amsthm}

\newtheorem*{claim*}{Claim}

\usepackage{xparse} 
\usepackage{etoolbox}
\NewDocumentCommand{\angles}{mmO{1}}{%
  \ensuremath{%
    \ifcase\numexpr#3\relax
      \PackageError{angles}{Nesting level must be >= 1}{You passed #3}%
    \or
      #1\langle #2 \rangle
    \or
      #1\langle\langle #2 \rangle\rangle
    \or
      #1\langle\langle\langle #2 \rangle\rangle\rangle
    \else
      \PackageError{angles}{Unsupported nesting level #3}{Only nesting levels 1--3 are supported}
    \fi
  }%
}
\usetikzlibrary{fit,matrix,positioning,backgrounds,shapes.geometric}
\newsavebox{\KTwoFourFigureBox}

\begin{document}
\title{The DP Color Function of Bipartite Graphs}

\author{ 
Hemanshu Kaul \thanks{Department of Applied Mathematics, Illinois Institute of Technology, Chicago, IL, USA
(kaul@illinoistech.edu)
}
\and
Jeffrey A. Mudrock \thanks{Department of Mathematics and Statistics, University of South Alabama, Mobile, AL, USA (mudrock@southalabama.edu)}
\and
Gunjan Sharma  \thanks{Department of Mathematics and Computer Science, Lake Forest College, Lake Forest, IL, USA
(gsharma@lakeforest.edu)}
\and
Anne Ullyot \thanks{Department of Applied Mathematics, Illinois Institute of Technology, Chicago, IL, USA
(aullyot@hawk.illinoistech.edu)}
}
\maketitle
\begin{abstract}
    DP-coloring (or correspondence coloring) is a generalization of list coloring that has been widely studied since its introduction by Dvo\v{r}\'{a}k and Postle in 2015. As the analogue of $P(G,q)$, the chromatic polynomial of a graph $G$, the DP color function of $G$, denoted by $P_{DP}(G,q)$, counts the minimum number of DP-colorings over all  $q$-fold covers of $G$. It follows that $P_{DP}(G,q) \leq P(G,q)$.  It is known that there are graphs for which $P_{DP}(G,q) < P(G,q)$ for all sufficiently large $q$; in fact, all bipartite graphs containing a cycle have this property.  A fundamental open question about DP color functions asks whether, for every graph $G$, there exist $N \in \N$ and a polynomial $p$ such that $P_{DP}(G,q) = p(q)$ whenever $q \geq N$.  In this paper we answer this question affirmatively for all bipartite graphs. Specifically, if $G$ is an $n$-vertex bipartite graph with $c$ components, then   $P_{DP}(G,q) = (-1)^{n-c}q^c \;T_G(1-q,1)$ for all sufficiently large $q$, where $T_G(x,y)$ is the Tutte polynomial of $G$. The ideas we develop also yield an asymptotic formula for $P(G,q)-P_{DP}(G,q)$ whenever the girth of $G$ is even.



\end{abstract}

\section{Introduction}\label{intro}

In 1990, Kostochka and Sidorenko~\cite{AS90} posed a natural question that would serve as the springboard for a broader study of the asymptotic behavior of enumerative functions arising from extensions of classical graph coloring.  They asked whether the list color function $P_\ell(G,q)$, which counts the minimum number of list colorings of $G$ over all $q$-list-assignments, equals the corresponding chromatic polynomial $P(G,q)$ when $q$ is large enough.  Donner~\cite{D92} answered this question affirmatively in 1992, and subsequent work improved the known bounds on how large $q$ must be to achieve this equality \cite{DZ22, T09, WQ17}. Whether this phenomenon, namely, {\sl equality between an enumerative function arising from a variant of list coloring and its classical counterpart whenever the number of colors is sufficiently large}, occurs has also been studied in the context of colorings of signed graphs~\cite{HQW24}, packings of list colorings~\cite{KM24}, and list coloring of unlabeled graphs~\cite{KM25}.

For DP-coloring, the focus of this paper and a common generalization of many notions of graph coloring, including list coloring (see~\cite{DKM23}), this phenomenon does not occur for all graphs.  For example, it is known that for every graph $G$ with girth that is even, the DP color function $P_{DP}(G,q)$, which counts the minimum number of DP-colorings of $G$ over all $q$-fold covers, does not equal $P(G,q)$ for all sufficiently large $q$. In this paper, we show a modified form of this phenomenon holds for the DP color function in the case of bipartite graphs. Instead of comparing $P_{DP}(G,q)$ with the chromatic polynomial, we show that $P_{DP}(G,q)$ equals a multiple of $T_G(1-q,1)$ for large enough $q$, where $T_G(x,y)$ is the Tutte polynomial of $G$. This represents major progress towards the conjecture that for every graph $G$, $P_{DP}(G,q)$ equals a polynomial (not necessarily $P(G,q)$) for sufficiently large $q$.

The remainder of this section provides the necessary definitions and an overview of our results.

\subsection{Preliminaries and Graph Coloring} \label{basic}

In this paper all graphs are nonempty, finite, simple graphs unless otherwise noted.  Generally speaking we follow West~\cite{W01} for terminology and notation.  The set of natural numbers is $\N = \{1,2,3, \ldots \}$. For $q\in\N$, we write $\Z_q$ for the set of integers modulo $q$. For $m\in\N$, we write $[m]=\{1,\ldots,m\}$, and we set $[0]=\emptyset$.  For $a,b \in \Z$, we let $[a:b] =\{i\in\Z:a\leq i\leq b\}$. Thus $[a:b]=\emptyset$ whenever $a>b$. We use the conventions that an empty sum is $0$, an empty product is $1$, an empty union is $\emptyset$, and the intersection of an empty family of subsets of a fixed ambient set is that ambient set. We also take $0^0=1$.

 If $G$ is a graph and $S, U \subseteq V(G)$, we use $G[S]$ for the subgraph of $G$ induced by $S$, and we use $E_G(S, U)$ for the set consisting of all the edges in $E(G)$ such that one endpoint is in $S$ and the other is in $U$.  If an edge in $E(G)$ connects the vertices $u$ and $v$, the edge can be represented by $uv$ or $vu$.  For $A \subseteq E(G)$, the spanning subgraph of $G$ with edge set $A$ is denoted by $\angles{G}{A}$.  We write $\angles{G}{A}[2]$ for the subgraph of $\angles{G}{A}$ obtained by deleting all isolated vertices. For a graph $G$, $\operatorname{comp}(G)$ denotes the number of components of $G$.

In the classical vertex coloring problem we wish to color the vertices of a graph $G$ with up to $q$ colors from $[q]$ so that adjacent vertices receive different colors; such a coloring is called a \emph{proper $q$-coloring}. The chromatic number of a graph $G$, denoted $\chi(G)$, is the smallest $q$ such that $G$ has a proper $q$-coloring.  List coloring, a well-known variation on classical vertex coloring, was introduced independently by Vizing~\cite{V76} and Erd\H{o}s, Rubin, and Taylor~\cite{ET79} in the 1970s.  For list coloring, we associate a \emph{list assignment} $L$ with a graph $G$ such that each vertex $v \in V(G)$ is assigned a list of colors $L(v)$ (we say $L$ is a list assignment for $G$).  Then, $G$ is \emph{$L$-colorable} if there exists a proper coloring $f$ of $G$ such that $f(v) \in L(v)$ for each $v \in V(G)$ (we refer to such an $f$ as a \emph{proper $L$-coloring} of $G$).  A list assignment $L$ is called a \emph{$q$-assignment} for $G$ if $|L(v)|=q$ for each $v \in V(G)$.  The \emph{list chromatic number} of a graph $G$, denoted $\chi_\ell(G)$, is the smallest $q$ such that $G$ is $L$-colorable whenever $L$ is a $q$-assignment for $G$.  We say $G$ is \emph{$q$-choosable} if $q \geq \chi_\ell(G)$.  Since $G$ must be $L$-colorable whenever $L$ is a $\chi_\ell(G)$-assignment for $G$ that assigns the same list of colors to each element in $V(G)$, it is clear that $\chi(G) \leq \chi_\ell(G)$.  This inequality may be strict since it is known that there are bipartite graphs with arbitrarily large list chromatic number (see~\cite{ET79}).

In 2015, Dvo\v{r}\'{a}k and Postle~\cite{DP15} introduced a generalization of list coloring called DP-coloring (they called it correspondence coloring) in order to prove that every planar graph without cycles of lengths 4 to 8 is 3-choosable. DP-coloring has been studied extensively\footnote{According to MathSciNet, the paper~\cite{DP15} by Dvo\v{r}\'{a}k and Postle has over 140 citations currently.}. Intuitively, DP-coloring is a variation of list coloring in which each vertex still receives a list of colors, but the correspondence between colors at adjacent vertices may vary from edge to edge.  We now give the formal definition.  Suppose $G$ is a graph.  
A \emph{cover} of $G$ is a pair $\mathcal{H} = (L,H)$ satisfying the following conditions.
        \begin{itemize}
            \item[{\bf -}] $H$ is a graph and $L$ is a function assigning to each $v \in V(G)$ a set  $L(v)$ satisfying $L(v) \subseteq V(H)$.
            \item[{\bf -}] The sets $L(v)$ for $v \in V(G)$ are pairwise disjoint, each is an independent set in $H$, and  $V(H) = \bigcup_{v \in V(G)} L(v)$. 
            \item[{\bf -}] For all distinct $u,v\in V(G)$, $E_H(L(u),L(v))$ is a matching, and it is empty whenever $uv\notin E(G)$.
        \end{itemize}
We stress that the matchings between $L(u)$ and $L(v)$ for $uv \in E(G)$ need not be perfect (and may even be empty). For $q\in\N$, we say that $\mathcal{H}$ is a \emph{$q$-fold} cover of $G$ if $|L(v)|=q$ for every $v\in V(G)$.  A $q$-fold cover $\mathcal{H}$ is said to be \emph{full} if for each $uv \in E(G)$, the matching $E_H(L(u),L(v))$ is perfect.  For any $q$-fold cover $\mathcal{H}$ of $G$ and any $A \subseteq E(G)$, the subcover of $\mathcal{H}$ restricted to $\angles{G}{A}$ is $\angles{\mathcal{H}}{A} = (L_{\angles{}{A}},H_{\angles{}{A}})$, where $L_{\angles{}{A}}= L$, $V(H_{\angles{}{A}}) = \bigcup_{v \in V(G)}L(v)$, and  $E(H_{\angles{}{A}}) = \bigcup_{uv \in A}E_H(L(u), L(v))$. The subcover of $\mathcal{H}$ restricted to $\angles{G}{A}[2]$ is ${\angles{\mathcal{H}}{A}[2]} = (L_{\angles{}{A}[2]}, H_{\angles{}{A}[2]})$, where $L_{\angles{}{A}[2]}(u) = L(u)$ for all $u \in V(\angles{G}{A}[2])$, $V(H_{\angles{}{A}[2]}) = \bigcup_{u \in V(\angles{G}{A}[2])}L(u)$, and $E(H_{\angles{}{A}[2]}) = \bigcup_{uv \in A}E_H(L(u), L(v))$.

Suppose $G$ is a graph and $\mathcal{H}= (L,H)$ is a $q$-fold cover of $G$. A \emph{transversal} of $\mathcal{H}$ is a set of vertices $T \subseteq V(H)$ such that $|T \cap L(v)|=1$ for all $v \in V(G)$.  We let $\mathcal{T}_{\mathcal{H}}$ be the set of all transversals of $\mathcal{H}$.  A transversal is called \emph{independent} if it is an independent set in $H$. A \emph{proper $\mathcal{H}$-coloring} of $G$ is an independent transversal of $\mathcal{H}$.  We let $\mathcal{I}_{\mathcal{H}}$ be the set of all  proper $\mathcal{H}$-colorings of $G$. An \emph{anti-$\mathcal{H}$-coloring} of $G$ is a transversal $T$ of $\mathcal{H}$ such that $H[T]$ is isomorphic to $G$. We let $\mathcal{F}_{\mathcal{H}}$ be the set of all anti-$\mathcal{H}$-colorings of $G$.

Suppose $\mathcal{H} = (L,H)$ is a $q$-fold cover of $G$.  We say  $\mathcal{H}$ 
is \emph{canonical} if it admits a \emph{canonical labeling}, that is, a mapping $\lambda \colon V(H) \to [q]$ such that the following conditions are satisfied. 
\begin{itemize}
    \item[{\bf -}] For every $v \in V(G)$, the restriction of $\lambda$ to $L(v)$ is a bijection from $L(v)$ to $[q]$.
    \item[{\bf -}] For all $uv \in E(G)$, $c \in L(u)$, and $c' \in L(v)$, we have $cc' \in E(H)$ if and only if $\lambda(c) = \lambda(c')$.
\end{itemize}
Now suppose $\mathcal{H} = (L,H)$ is a canonical $q$-fold cover of $G$ with canonical labeling $\lambda$. For each $v\in V(G)$ and $j \in [q]$, let $(v,j)$ denote the unique vertex of $L(v)$ satisfying $\lambda((v,j))=j$.  With this notation, for every
$uv\in E(G)$ and $i,j\in[q]$, $(u,i)(v,j)\in E(H)$ if and only if $i =j$. Let $\mathcal{C}$ be the set of proper $q$-colorings of $G$. The function $f: \mathcal{C} \to \mathcal{I}_{\mathcal{H}}$ given by $f(c) = \{ (v,c(v)) : v \in V(G) \}$ is a bijection. Moreover, for each $j \in [q]$, $H[\bigcup_{v \in V(G)}\{(v,j)\}]$ is isomorphic to $G$. Thus $\bigcup_{v \in V(G)}\{(v,j)\} \in \mathcal{F}_{\mathcal{H}}$. It is also important to note that over all $q$-fold covers $\mathcal{H}$ of $G$, $|\mathcal{F}_{\mathcal{H}}|$ is maximized precisely when $\mathcal{H}$ is canonical (see Lemma~\ref{lem:toolbox} below).

\subsection{Counting Colorings and Our Motivation}

In 1912, Birkhoff~\cite{B12} introduced the notion of the chromatic polynomial in hopes of using it to make progress on the four color problem.  For $q \in \N$, the \emph{chromatic polynomial} of a graph $G$, denoted $P(G,q)$, is the number of proper $q$-colorings of $G$. It is easy to show that $P(G,q)$ is a polynomial in $q$ of degree $|V(G)|$ (see~\cite{DKT05}). For example, whenever $n, q \in \N$ it is well known that $P(K_n,q) = \prod_{i=0}^{n-1} (q-i)$, $P(C_n,q) = (q-1)^n + (-1)^n (q-1)$, and $P(T,q) = q(q-1)^{n-1}$ whenever $T$ is a tree on $n$ vertices, (see~\cite{B94, W01}).

Beyond graph coloring itself, chromatic polynomials and their generalizations are closely connected to the Tutte polynomial, one of the central invariants of modern combinatorics. Through these connections, enumerative coloring questions interact
with matroid theory~\cite{B72, C69}, hyperplane arrangements~\cite{GZ83}, statistical physics, and knot theory~\cite{W94}.
Suppose $G$ is an $n$-vertex graph. The \emph{Tutte polynomial} of $G$, denoted $T_G(x,y)$, is given by
\[
    T_G(x,y)
    =
    \sum_{A\subseteq E(G)}
    (x-1)^{\operatorname{comp}(\angles{G}{A})-\operatorname{comp}(G)}
    (y-1)^{|A|-n+\operatorname{comp}(\angles{G}{A})}.
\]
It is well known that $P(G,q)=(-1)^{n-\operatorname{comp}(G)}q^{\operatorname{comp}(G)}T_G(1-q,0)$ (see~\cite{EM10}).  In this paper, we show another specialization of the Tutte polynomial obtained from this formula by replacing $y=0$ with $y=1$ plays an essential role in establishing a lower bound on the guaranteed number of DP-colorings for a bipartite graph when $q$ is sufficiently large (see Remark~\ref{rem:Tutte}).

The notion of chromatic polynomial was extended to list coloring in the early 1990s by Kostochka and Sidorenko~\cite{AS90}.  If $L$ is a list assignment for a graph $G$, let $P_\ell(G,L)$ denote the number of proper $L$-colorings of $G$. For $q \in \N$, the \emph{list color function} of $G$, denoted $P_\ell(G,q)$, is the minimum value of $P_\ell(G,L)$ where the minimum is taken over all possible $q$-assignments $L$ for $G$.  Since a $q$-assignment could assign the same $q$ colors to every vertex in a graph, it is clear that $P_\ell(G,q) \leq P(G,q)$ for each $q \in \N$.  In general, the list color function can differ significantly from the chromatic polynomial for small values of $q$.  For example, for any $n \geq 2$, $P_{\ell}(K_{n,n^n},q) = 0$ and $P(K_{n,n^n},q) > 1$ whenever $q \in [2:n]$~\cite{ET79}.  On the other hand, in 2023 Dong and Zhang~\cite{DZ22}, improving on earlier results (see~\cite{D92, T09, WQ17}), showed that for any graph $G$, $P_\ell(G,q) = P(G,q)$ whenever $q \geq |E(G)|-1$.  

In 2019, the first and second named authors introduced a DP-coloring analogue of the chromatic polynomial in hopes of gaining a better understanding of DP-coloring and using it as a tool for making progress on some open questions related to the list color function~\cite{KM19}.  Since 2019 it has received some attention in the literature (see for example~\cite{BB23, BK23, DKM22, DY21, DZ23, LY22, M21, MT20} and references therein).  Specifically, suppose $\mathcal{H} = (L,H)$ is a cover of graph $G$.  Let $P_{DP}(G, \mathcal{H})$ be the number of $\mathcal{H}$-colorings of $G$.  Then, the \emph{DP color function} of $G$, denoted $P_{DP}(G,q)$, is the minimum value of $P_{DP}(G, \mathcal{H})$ where the minimum is taken over all  $q$-fold covers $\mathcal{H}$ of $G$.  
Note that this minimum may equivalently be taken over all full $q$-fold covers. Thus, whenever a $q$-fold cover $\mathcal{H}= (L,H)$ such that $P_{DP}(G,\mathcal{H})= P_{DP}(G,q)$ is assumed to exist, we may assume that the cover is full.  It is easy to prove that
\[
P_{DP}(G,q)\le P_\ell(G,q)\leq P(G,q)
\]
for each $q \in \N$.
Note that if $G$ is a disconnected graph with components $W_1, W_2, \ldots, W_k$, then $P_{DP}(G, q) = \prod_{i=1}^k P_{DP}(W_i,q)$.

The DP color function of certain graphs is identical to the corresponding list color function. For example, $P_{DP}(G,q) = P(G,q)$ for every $q \in \N$  whenever $G$ is chordal or an odd cycle (\cite{KM19}). On the other hand, some graphs have a DP color function that behaves in a very different manner than its list color function.  For example, it is known (\cite{KM19}) that if $G$ is a graph with girth that is even, there is an $N \in \N$ such that $P_{DP}(G,q) < P(G,q)$ whenever $q \geq N$. Asymptotically, we have the following result.

\begin{thm} [\cite{MT20}] \label{thm: general}
Suppose $g$ is an odd integer with $g \geq 3$.  If $G$ is a graph on $n$ vertices with girth $g$ or $g+1$, then $P(G,q) - P_{DP}(G,q) = O(q^{n-g})$ as $q \to \infty$.  Moreover, for each odd integer $g \geq 3$ there is a graph $G$ on $n$ vertices with girth $g+1$ such that $P(G,q) - P_{DP}(G,q) = \Theta(q^{n-g})$ as $q \to \infty$. 
\end{thm}
We improve Theorem~\ref{thm: general} by proving that for  every  graph $G$ with girth $g+1$ that is even, $P(G,q) - P_{DP}(G,q) = \Theta(q^{n-g})$ as $q \to \infty$. We also find the asymptotic leading coefficient of $P(G,q) - P_{DP}(G,q)$. 

Notice Theorem~\ref{thm: general} shows that asymptotically the DP color function of any graph remains close to its chromatic polynomial. This suggests the stronger possibility that, for every graph $G$, its DP color function may eventually agree exactly with some polynomial. This leads to the following fundamental open question.

\begin{ques} [\cite{KM19}] \label{ques: main}
For every graph $G$, is there an $N \in \N$ and a polynomial $p(q)$ such that $P_{DP}(G, q) = p(q)$ whenever $q \geq N$?
\end{ques}
Question~\ref{ques: main} has an affirmative answer whenever $G$ has a vertex $v$ such that $G-v$ is acyclic; in particular, this class contains all generalized theta graphs~\cite{HK21}. It is also affirmative when $G$ has a dominating vertex; in this case, the stronger conclusion $P_{DP}(G,q)=P(G,q)$ holds for all sufficiently large $q$~\cite{MT20}. In this paper, we answer Question~\ref{ques: main} affirmatively for every bipartite graph.

\subsection{Results and Outline of the Paper}
In Section~\ref{tools}, we develop several tools that apply to arbitrary graphs and will be used throughout the paper.    Importantly, Section~\ref{tools} ends with a proof that for any graph $G$ and sufficiently large $q$ there are $q$-fold covers of $G$ that are \emph{cycle-shattering}. We call a $q$-fold cover $\mathcal{H}$ a \emph{cycle-shattering cover} of $G$ if for all $A\subseteq E(G)$ such that $\angles{G}{A}$ contains a cycle, $|\mathcal{F}_{\angles{\mathcal{H}}{A}}| =0$.  Cycle-shattering covers are key to the proofs of Theorems~\ref{thm: bipartite} and \ref{thm:evencycles}.

In Section~\ref{bipartite}, we answer Question~\ref{ques: main} affirmatively for bipartite graphs by proving the following result. 
\begin{thm}\label{thm: bipartite}
    Suppose $G$ is a bipartite graph with $n$ vertices and $m$ edges. If $q \geq m + 2^m$ and $a_i$ is the number of acyclic spanning subgraphs of $G$ with $i$ components, then 
    \[
        P_{DP}(G,q) = \sum_{j=0}^{n-1}(-1)^ja_{n-j}q^{n-j}.
    \]
\end{thm}
\begin{rem}\label{rem:Tutte} 
    Note that  $\sum_{j=0}^{n-1}(-1)^ja_{n-j}q^{n-j}$ is exactly $(-1)^{n-\operatorname{comp}(G)}q^{\operatorname{comp}(G)}T_G(1-q,1)$ for all $q$. Recall that $P(G,q)=(-1)^{n-\operatorname{comp}(G)}q^{\operatorname{comp}(G)}T_G(1-q,0)$. Thus for bipartite graphs $G$ that contain a cycle, the eventual polynomial representing $P_{DP}(G,q)$ differs from $P(G,q)$, although both arise from specializations of the Tutte polynomial. This difference is further studied in Theorem~\ref{thm:evencycles} below for all graphs with girth that is even.
\end{rem}
Note that if the answer to Question~\ref{ques: main} is yes for a given connected graph $G$, $p(q)$ does not necessarily equal $(-1)^{n-1}qT_G(1-q,y)$ for any fixed $y \in \R$. With some computation, one can see the polynomial representation of $P_{DP}(G,q)$ given by Theorem 6 part (ii) in \cite{HK21} applied to the Theta graph $G =\Theta(1,2,3)$ gives a counterexample.\footnote{ A \emph{Theta Graph} $G= \Theta(\ell_1,\ell_2,\ell_3)$ consists of a pair of end vertices joined by $3$ internally disjoint paths of lengths $\ell_1, \ell_2,\ell_3 \in \N$.}

We make no attempt to optimize our bound on $q$ in Theorem~\ref{thm: bipartite}.  It would be interesting to determine how low one could take the bound on $q$ in Theorem~\ref{thm: bipartite}.  As a starting point for the study of such an optimal lower bound, we determine an exact formula for $P_{DP}(K_{2,n},q)$ whenever $q \in \N$.

\begin{thm} \label{thm:k2n}
Suppose $q, n \in \N$ and $p$, $r$ are nonnegative integers satisfying $qn=q^{2}p+r$ where $p = \lfloor{n}/{q}\rfloor$. Then, $P_{DP}(K_{2,n},q) = \sum_{i=1}^{q^{2}}(q-1)^{a_{i}}(q-2)^{n-a_{i}}$ where, for each $j \in [q^2]$,
\[
    a_j=
    \begin{cases}
    p+1 & \text{ if } j \leq r \\
    p & \text{ if } j > r.
    \end{cases}
\]
\end{thm}
It follows from Theorem~\ref{thm:k2n} that, for all $q,n\in\N$,
$P_{DP}(K_{2,n},q)$ is identical to $(-1)^{n+1}qT_{K_{2,n}}(1-q,1)$ (the formula in Theorem~\ref{thm: bipartite}) if and only if $q\geq |V(K_{2,n})|-2$ or $q=2$. It is natural to ask whether the threshold, $|E(G)|+2^{|E(G)|}$, on $q$ in Theorem~\ref{thm: bipartite} can be lowered to $O(|V(G)|)$ or even a polynomial in $|V(G)|$. Such questions on bounding the threshold on number of colors needed to observe the equality between list color function and the chromatic polynomial have been well-studied in the literature, see~\cite{DZ22, KK22, T09, WQ17}.

Finally in Section~\ref{girth}, we improve on the upper bound on $P(G,q) - P_{DP}(G,q)$ given in Theorem~\ref{thm: general} by giving an asymptotically sharp formula for this difference when $G$ is a graph with girth that is even. 
\begin{thm} \label{thm:evencycles}
     Let $g \geq 3$ be odd. Suppose $G$ is an $n$-vertex graph with $m$ edges and girth $g+1$, and let $t$ be the number of $(g+1)$-cycles in $G$. Then 
     \[
        P(G,q) - P_{DP}(G,q) = t q^{n-g} + O(q^{n-g-1})
     \]
     as $q \to \infty$.  
\end{thm}
If $G$ has girth $g+1$ where $g+1$ is even, and $P_{DP}(G,q)$ is given by a polynomial $p(q)$ for all sufficiently large $q$, then Theorem~\ref{thm:evencycles} determines the coefficients of the $g+1$ highest-order terms of $p(q)$.

\section{General Tools}\label{tools}
In this section we present some important tools that will be used throughout the paper and may be of independent interest as they hold for all graphs.  Crucially, the following inclusion--exclusion formula is well-known.
\begin{lem}\label{lem: PIE}(\cite{MT20})
    For any graph $G$ and any $q$-fold cover $\mathcal{H} = (L,H)$ of $G$, 
    \[
        P_{DP}(G,\mathcal{H})
        = \sum_{A \subseteq E(G)} (-1)^{|A|}\, \bigl|\mathcal{F}_{\angles{\mathcal{H}}{A}}\bigr|.
    \]
\end{lem}
Lemma~\ref{lem: PIE} makes explicit the role anticolorings play in computing the number of proper $\mathcal{H}$-colorings of $G$. Consequently, we develop in Lemma~\ref{lem:toolbox} some basic properties of $\mathcal{F}_{\angles{\mathcal{H}}{A}}$ when $A \subseteq E(G)$. We conclude the section by proving in Lemma~\ref{lem: shattering_existence} that every graph admits a cycle-shattering $q$-fold cover whenever $q$ is sufficiently large.  Intuitively, one might expect cycle-shattering covers to play an important role in studying DP color functions of bipartite graphs, since for every even cycle $C$ and every $q\geq2$, any $q$-fold cover $\mathcal{H}$ of $C$ satisfying $P_{DP}(C,\mathcal{H})=P_{DP}(C,q)$ has no anti-$\mathcal{H}$-colorings of $C$~\cite{KM19}.

Throughout this section, suppose $G$ is an $n$-vertex graph with $m$ edges and that $\mathcal{H} = (L,H)$ is a $q$-fold cover of $G$. For $A\subseteq E(G)$, let $U(A)$ denote the set of vertices incident to an edge of $A$. For each $uv \in E(G)$, we write $\mathcal{E}_\mathcal{H}(uv) = \{ T \in \mathcal{T}_\mathcal{H}: |E(H[T]) \cap E_H(L(u), L(v))| =1\}$. Note that $\mathcal{E}_\mathcal{H}(uv)$ is the set of transversals of $\mathcal{H}$ where the edge $uv$ prevents the transversal from being independent; in fact, $\mathcal{E}_\mathcal{H}(uv) = \mathcal{F}_{\angles{\mathcal{H}}{\{uv\}}}$.  The following observation is immediate.
\begin{obs}  \label{obs: F_equals_E_intersection}(\cite{Sharma24})
    For any graph $G$ and cover $\mathcal{H}$ of $G$,
    $\mathcal{F}_{\mathcal{H}} = \bigcap_{uv \in E(G)} \mathcal{E}_{\mathcal{H}}(uv)$.
\end{obs}
We next collect several properties of anti-colorings that will be used repeatedly.
\begin{lem}\label{lem:toolbox}
Let $G$ be a graph and $\mathcal{H}$ be a $q$-fold cover of $G$. Fix $A \subseteq E(G)$. Then the following statements hold.
\begin{enumerate}[(i)]
    \item If $A \neq \emptyset$ and $\angles{G}{A}[2]$ is connected, then for any  distinct $F,F'\in \mathcal{F}_{\angles{\mathcal H}{A}[2]}$, $F\cap L(v)\neq F'\cap L(v)$ for every $v\in U(A)$. \label{tool:intersection_L(v)}
    \item For every $A'\subseteq A$,  $|\mathcal{F}_{\angles{\mathcal{H}}{A'}}|\geq |\mathcal{F}_{\angles{\mathcal{H}}{A}}|$. \label{tool:edgesubset_single}
    \item  If $\emptyset\neq A'\subseteq A$ and $\angles{G}{A}[2]$ is connected, then $|\mathcal{F}_{\angles{\mathcal{H}}{A'}[2]}|\geq |\mathcal{F}_{\angles{\mathcal{H}}{A}[2]}|$. 
    \label{tool:edgesubset_double} 
    \item For every component $W$ of $\angles{G}{A}[2]$, $|\mathcal{F}_{\angles{\mathcal{H}}{E(W)}[2]}| \leq q$,
    with equality if and only if $\angles{\mathcal{H}}{E(W)}[2]$ admits a canonical labeling.\label{tool:canon} 
    \item Let $W_1,\dots,W_k$ be the components of $\angles{G}{A}[2]$, where $k$ may be zero.  Then
    \[
    |\mathcal{F}_{\angles{\mathcal{H}}{A}}| = q^{n-|U(A)|}\prod_{i=1}^{k}|\mathcal{F}_{\angles{\mathcal{H}}{E(W_i)}[2]}| \leq q^{\operatorname{comp}({\angles{G}{A}})}.
    \]
    Consequently, $|\mathcal{F}_{\angles{\mathcal{H}}{A}}|=0$ if and only if there exists  $i \in [k]$ such that  $|\mathcal{F}_{\angles{\mathcal{H}}{E(W_i)}[2]}| = 0$. \label{tool:product} 
\end{enumerate}
\end{lem}
\begin{proof}
\emph{(i)}  Suppose for contradiction that $F\cap L(u)=F'\cap L(u)$ for some $u\in U(A)$.  As $F\neq F'$, there exists $w\in U(A)$ such that $F\cap L(w)\neq F'\cap L(w)$.  Since $\angles{G}{A}[2]$ is connected, there is a $u,w$-path in $\angles{G}{A}[2]$.  This implies there exist adjacent $a,b\in U(A)$ along this path with $F\cap L(a)=F'\cap L(a)$ and $F\cap L(b)\neq F'\cap L(b)$. Let $x$ be the unique element in $F\cap L(a)=F'\cap L(a)$. Because $E_H(L(a),L(b))$ is a matching, $x$ is adjacent to at most one vertex in $L(b)$. Thus at least one of $F,F'$ does not belong to $\mathcal{E}_{\angles{\mathcal{H}}{A}[2]}(ab)$. Since $ab\in A$, Observation~\ref{obs: F_equals_E_intersection}, applied to $\angles{\mathcal{H}}{A}[2]$, implies that at least one of $F,F'$ does not belong to $\mathcal{F}_{\angles{\mathcal{H}}{A}[2]}$, a contradiction.
\\
\\
\emph{(ii)} By Observation~\ref{obs: F_equals_E_intersection}, 
\[
\mathcal{F}_{\angles{\mathcal{H}}{A}} = \bigcap_{uv \in E(\angles{G}{A})}\mathcal{E}_\mathcal{H}(uv) \subseteq \bigcap_{uv \in E(\angles{G}{A'})}\mathcal{E}_\mathcal{H}(uv) = \mathcal{F}_{\angles{\mathcal{H}}{A'}}.
\]
Thus $|\mathcal{F}_{\angles{H}{A'}}| \geq  |\mathcal{F}_{\angles{H}{A}}|$.
\\
\\
\emph{(iii)}   Consider the  map $\phi: \mathcal{F}_{\angles{\mathcal{H}}{A}[2]} \to \mathcal{F}_{\angles{\mathcal{H}}{A'}[2]}$ defined by $\phi(F) = F \cap \bigcup_{v \in U(A')}L(v)$. It is easy to see $\phi$ is a function. By Statement~(\ref{tool:intersection_L(v)}), $\phi$ is injective, and the result follows.
\\
\\
\emph{(iv)} Fix a component $W$ of $\angles{G}{A}[2]$ and suppose, for a contradiction, that $|\mathcal{F}_{\angles{\mathcal{H}}{E(W)}[2]}|>q$. Fix $u\in V(W)$. Since $\mathcal{H}$ is $q$-fold, $|L(u)|=q$. Every element of
$\mathcal{F}_{\angles{\mathcal{H}}{E(W)}[2]}$ is a transversal and therefore contains exactly one vertex of $L(u)$. Thus there exist distinct $F,F'\in\mathcal{F}_{\angles{\mathcal{H}}{E(W)}[2]}$
such that $F\cap L(u)=F'\cap L(u)$. This contradicts Statement~(\ref{tool:intersection_L(v)}) applied with $A= E(W)$. Therefore, $|\mathcal{F}_{\angles{\mathcal{H}}{E(W)}[2]}|\leq q$.
    
Now assume $|\mathcal{F}_{\angles{\mathcal{H}}{E(W)}[2]}|=q$ and write $\mathcal{F}_{\angles{\mathcal{H}}{E(W)}[2]}=\{F_1,\dots,F_q\}$. By Statement~(\ref{tool:intersection_L(v)}),
$F_i\cap L(v)\neq F_j\cap L(v)$ whenever $i\neq j$ and $v\in V(W)$. Since $|L(v)|=q$, it follows that $\{F_1\cap L(v),\dots,F_q\cap L(v)\}$ is a partition of $L(v)$ for every $v\in V(W)$. For $v\in V(W)$ and $x\in L(v)$, define $\lambda(x)$ to be the unique index $i\in[q]$ such that $x\in F_i$. The restriction of $\lambda$ to $L(v)$ is therefore a bijection from $L(v)$ to $[q]$.

Fix $uv\in E(W)$, $c\in L(u)$, and $c'\in L(v)$. If $\lambda(c)=\lambda(c')=i$, then $c,c'\in F_i$, so Observation~\ref{obs: F_equals_E_intersection}, applied to $W$ and its cover $\angles{\mathcal{H}}{E(W)}[2]$, implies that $cc'\in E(H)$. Conversely, suppose that $cc'\in E(H)$ and let $\lambda(c)=i$. Let $d$ be the unique vertex in $F_i\cap L(v)$. Since $\lambda(c)=\lambda(d)=i$, the preceding implication gives $cd\in E(H)$. Because $E_H(L(u),L(v))$ is a matching, $c'=d$, and hence $\lambda(c')=i=\lambda(c)$. Therefore, $cc'\in E(H)$ if and only if $\lambda(c)=\lambda(c')$, so $\lambda$ is a canonical labeling of $\angles{\mathcal{H}}{E(W)}[2]$.

Conversely, assume  $\angles{\mathcal{H}}{E(W)}[2]$ admits a canonical labeling. This allows us to name the vertices of $H_{\angles{}{E(W)}[2]}$  such that for each $v \in V(W)$, $L(v) = \{(v,j) : j \in [q]\}$ and $(u,j)(v,j) \in E(H_{\angles{}{E(W)}[2]})$ whenever $uv \in E(W)$. In particular, for each $j \in [q]$, the set $F_j = \{ (v,j) : v \in V(W) \}$ is a transversal of $\angles{\mathcal{H}}{E(W)}[2]$, and $H_{\angles{}{E(W)}[2]}[F_j]$ is isomorphic to $\angles{G}{E(W)}[2]$. Hence $F_j \in \mathcal{F}_{\angles{\mathcal{H}}{E(W)}[2]}$ for all $j \in [q]$, implying $|\mathcal{F}_{\angles{\mathcal{H}}{E(W)}[2]}| \ge q$. Together with the previous upper bound, this yields $|\mathcal{F}_{\angles{\mathcal{H}}{E(W)}[2]}| = q$.
\\
\\
 \emph{(v)} 
 Let $Z=V(G)- U(A)$. First we assume $k \geq 1$ and $Z \neq \emptyset$. Write $Z= \{z_1, \dots, z_p\}$.  Define the function $\phi: \mathcal{F}_{\angles{\mathcal{H}}{A}} \to (\prod_{i=1}^k \mathcal{F}_{\angles{\mathcal{H}}{E(W_i)}[2]}) \times (\prod_{z \in Z}\{\{x\}: x \in L(z)\})$ by  $\phi(F) = \big( (F_i, \dots, F_k),  (F \cap L(z_1), \dots, F \cap L(z_p)) \big)$ where $F_i = F \cap \big(\bigcup_{v \in V(W_i)}L(v)\big)$.  By Observation~\ref{obs: F_equals_E_intersection}, each $F_i\in\mathcal{F}_{\angles{\mathcal{H}}{E(W_i)}[2]}$, so $\phi$ is a function.

Suppose $F,F'\in\mathcal{F}_{\angles{\mathcal{H}}{A}}$ and $\phi(F)=\phi(F')$. Then $F_i=F_i'$ for every $i\in[k]$ and $F\cap L(z_j)=F'\cap L(z_j)$ for every $j\in[p]$.
Consequently, $F = (\bigcup_{i=1}^kF_i) \cup (\bigcup_{j=1}^p(F \cap L(z_j)))=(\bigcup_{i=1}^kF'_i) \cup (\bigcup_{j=1}^p(F' \cap L(z_j))) = F'$. Thus $\phi$ is injective.

Let $\big((F_1,\dots,F_k),(\{x_{z_1}\},\dots,\{x_{z_p}\})\big)$ be an arbitrary element of the codomain of $\phi$, and define $F=\left(\bigcup_{i=1}^kF_i\right)\cup\left(\bigcup_{j=1}^p\{x_{z_j}\}\right)$. Since $\{V(W_1),\dots,V(W_k), Z\}$ is a partition $V(G)$, the set $F$ is a transversal of $\angles{\mathcal{H}}{A}$.

Let $uv\in A$, and let $W_i$ be the component of $\angles{G}{A}[2]$ containing $u$ and $v$. Then
$F\cap\left(\bigcup_{w\in V(W_i)}L(w)\right)=F_i$. Since $F_i\in\mathcal{F}_{\angles{\mathcal{H}}{E(W_i)}[2]}$, Observation~\ref{obs: F_equals_E_intersection}, applied to $W_i$ and the cover $\angles{\mathcal{H}}{E(W_i)}[2]$, implies that the unique elements of $F\cap L(u)$ and $F\cap L(v)$ are adjacent in $H_{\angles{}{E(W_i)}[2]}$, and hence in $H_{\angles{}{A}}[F]$.

Now suppose $u,v\in V(G)$ are distinct and $uv\notin E(\angles{G}{A})$. By the definition of $H_{\angles{}{A}}$, there is no edge between $L(u)$ and $L(v)$ in $H_{\angles{}{A}}$. Therefore, the map sending each $v\in V(G)$ to the unique element of $F\cap L(v)$ is an isomorphism from $\angles{G}{A}$ to $H_{\angles{}{A}}[F]$. Hence $F\in\mathcal{F}_{\angles{\mathcal{H}}{A}}$ and
$\phi(F)=\big((F_1,\dots,F_k),(\{x_{z_1}\},\dots,\{x_{z_p}\})\big)$.
Thus $\phi$ is surjective.

Since $\phi$ is bijective, $|\mathcal{F}_{\angles{\mathcal{H}}{A}}| = \Bigl(\prod_{i=1}^k \bigl|\mathcal F_{\angles{\mathcal H}{E(W_i)}[2]}\bigr|\Bigr)
\cdot
\Bigl(\prod_{z\in Z} \bigl|\{\{x\}:x\in L(z)\}\bigr|\Bigr)$. Also, since $\bigl|\{\{x\}:x\in L(z)\}\bigr|=|L(z)|=q$ for each $z\in Z$, $\bigl|\mathcal F_{\angles{\mathcal H}{A}}\bigr|
=
q^{n-|U(A)|}\prod_{i=1}^k \bigl|\mathcal F_{\angles{\mathcal H}{E(W_i)}[2]}\bigr|$.  By Statement~(\ref{tool:canon}), for each $i \in [k]$, $\bigl|\mathcal F_{\angles{\mathcal H}{E(W_i)}[2]}\bigr|\le q$. Also, since $\operatorname{comp}(\angles{G}{A}) = n- |U(A)| +k $, we have 
\[
\bigl|\mathcal F_{\angles{\mathcal H}{A}}\bigr|
=
 q^{n-|U(A)|}\prod_{i=1}^k \bigl|\mathcal F_{\angles{\mathcal H}{E(W_i)}[2]}\bigr|
\le
q^{\operatorname{comp}(\angles{G}{A})}.
\]
Moreover, $|\mathcal{F}_{\angles{\mathcal{H}}{A}}|=0$ if and only if there exists some $i \in [k]$ such that $|\mathcal{F}_{\angles{\mathcal{H}}{E(W_i)}[2]}|=0$.  When $k =0$ or $Z = \emptyset$, it is easy to adapt the proof above to obtain the desired result.
\end{proof}    
The following result is helpful in applying Lemma~\ref{lem:toolbox} in the case when $\angles{G}{A}$ is acyclic.
\begin{lem}(\cite{DKM23})\label{lem: acyclic}
    Suppose $\mathcal{H}= (L, H)$ is a full  $q$-fold cover of an acyclic graph $G$. Then $\mathcal{H}$ admits a canonical labeling. 
\end{lem}
Note that in Statement~(\ref{tool:edgesubset_double}) of Lemma~\ref{lem:toolbox}, the requirement that $\angles{G}{A}[2]$ is connected is necessary. For example, suppose $q \geq 2$ and let $\angles{G}{A}[2]$ be the disjoint union of two edges. If $A'$ is one edge within $A$, then by Lemma~\ref{lem: acyclic}, for any full $q$-fold cover $\mathcal{H} = (L,H)$ of $G$, $|\mathcal{F}_{\angles{\mathcal{H}}{A}[2]}| = q^2$ while $|\mathcal{F}_{\angles{\mathcal{H}}{A'}[2]}| = q$.\\

We next prove that every graph admits a cycle-shattering $q$-fold cover whenever $q$ is sufficiently large.
\begin{lem}\label{lem: shattering_existence}
    Suppose $G$ is a graph with $n$ vertices and $m$ edges. If $q\geq 2^m$, then there exists a full $q$-fold cycle-shattering cover $\mathcal{H}$ of $G$.
    Consequently, if $a_i$ is the number of acyclic spanning subgraphs of $G$ with $i$ components, then 
    \[
        P_{DP}(G,\mathcal{H}) = \sum_{j=0}^{n-1}(-1)^ja_{n-j}q^{n-j}.
    \]
\end{lem}
\begin{proof}
    Throughout this argument numerical quantities will be interpreted as elements of $\Z_q$ or elements of $\Z$ depending on context.  In the first part of the argument, addition is modulo $q$.  Then, for the remainder of the argument we use standard addition.  Suppose the edges of $G$ are $e_1, \dots, e_m$ and the vertices of $G$ are $v_1, \dots, v_n$. We construct a $q$-fold cover $\mathcal{H} = (L, H)$ of $G$ as follows.  For each $v_j \in V(G)$ let $L(v_j) = \{(j,i): i \in \Z_{q}\}$. It remains to specify $E(H)$. 
    
    For each $s\in [m]$, let $\pi_s: \Z_q \to \Z_q$ be the permutation given by $\pi_s(x) = x+ 2^{s-1}$. Note $\pi^{-1}_s(x) = x-2^{s-1}$.  For each $e_j \in E(G)$, if $e_j = v_\alpha v_\beta$ where $\alpha < \beta$, then for each $i \in \Z_q$, add the edge $(\alpha, i)(\beta,\pi_j(i))$ to $E(H)$. Hence $\mathcal{H}$ is a full $q$-fold cover of $G$ and this completes the construction of $H$.

    For the sake of contradiction, suppose there exists $B\subseteq E(G)$ such that $\angles{G}{B}$ contains a cycle and $|\mathcal{F}_{\angles{\mathcal{H}}{B}}|>0$. Choose a cycle $C$ in $\angles{G}{B}$. Suppose the vertices of $C$ in cyclic order are $v_{b_1},v_{b_2},\dots,v_{b_k}$, where $k\leq m$. By cyclic order we mean $V(C)=\{v_{b_1},v_{b_2},\dots,v_{b_k}\}$ and $E(C)=\{v_{b_i}v_{b_{i+1}}:i\in[k]\}$, where $b_{k+1}=b_1$. For each $i\in[k]$, let $a_i\in[m]$ be such that $e_{a_i}=v_{b_i}v_{b_{i+1}}$. Since the edges of $C$ are distinct, $a_1,\dots,a_k$ are pairwise distinct.

    Suppose $F$ is an anti-$\angles{\mathcal{H}}{B}$-coloring of $\angles{G}{B}$. For each $i\in[k]$, let $c_i\in\Z_q$ be such that $(b_i,c_i)$ is the unique element of $F\cap L(v_{b_i})$, and let $c_{k+1}=c_1$. By Observation~\ref{obs: F_equals_E_intersection}, applied to $\angles{G}{B}$ and its cover $\angles{\mathcal{H}}{B}$, the vertices $(b_i,c_i)$ and $(b_{i+1},c_{i+1})$ are adjacent in $H_{\angles{}{B}}$, and hence in $H$, for every $i\in[k]$. Thus $c_{i+1}=\pi_{a_i}(c_i)$ if $b_i<b_{i+1}$ and $c_{i+1}=\pi^{-1}_{a_i}(c_i)$ if $b_i>b_{i+1}$ for each $i \in [k]$.  For each $i\in[k]$, let $d_i=1$ if $b_i<b_{i+1}$ and $d_i=-1$ if $b_i>b_{i+1}$.  Since $\pi_{a_i}(x)=x+2^{a_i-1}$ and $\pi_{a_i}^{-1}(x)=x-2^{a_i-1}$, it follows for each $i \in [k]$ that  $c_{i+1} = c_i + d_i2^{a_i-1}$.  These equations imply 
    \[
        c_{k+1}=c_1+\sum_{i=1}^k d_i2^{a_i-1}.
    \]
    Since $c_{k+1}=c_1$, we have $\sum_{i=1}^k d_i2^{a_i-1}=0$.

    This means that if we interpret $d_i2^{a_i-1}$ as an integer for each $i \in [k]$ and use standard addition, $\sum_{i=1}^k d_i2^{a_i-1}$ must be divisible by $q$.  We will use this fact to obtain a contradiction.
    
    For the remainder of the proof we use standard addition. Since $d_i\in \{-1,1\}$ for every $i\in[k]$, $|\sum_{i=1}^k d_i2^{a_i-1}| \leq \sum_{i=1}^k 2^{a_i-1}$.  Since the indices $a_1,\ldots,a_k$ are pairwise distinct elements of $[m]$, \\ $2^{a_1-1},\ldots,2^{a_k-1}$ are pairwise distinct as well. Therefore, since $q \geq 2^m$,
    \[
        \left|\sum_{i=1}^k d_i2^{a_i-1}\right| \leq \sum_{i=1}^k 2^{a_i-1}\leq \sum_{s=1}^m 2^{s-1} =2^m-1<q.
    \]
    Thus $\sum_{i=1}^k d_i2^{a_i-1}$ is divisible by $q$ and satisfies $|\sum_{i=1}^k d_i2^{a_i-1}|<q$, which means $\sum_{i=1}^k d_i2^{a_i-1}=0$.  However, this is impossible, since moving the negative terms of $\sum_{i=1}^k d_i2^{a_i-1}$ to the right side of this equation implies there is a nonnegative integer with two distinct binary representations. Thus for any $A\subseteq E(G)$ such that $\angles{G}{A}$ contains a cycle, $|\mathcal F_{\angles{\mathcal H}{A}}|=0$.

   By Lemma~\ref{lem: PIE}, only the terms corresponding to sets $A\subseteq E(G)$ for which $\angles{G}{A}$ is acyclic can be nonzero. For every such $A$, the cover $\angles{\mathcal{H}}{A}$ is full, so Lemmas~\ref{lem: acyclic} and \ref{lem:toolbox}(\ref{tool:canon}, \ref{tool:product}) imply that $|\mathcal{F}_{\angles{\mathcal{H}}{A}}|=q^{\operatorname{comp}(\angles{G}{A})}$. Since an acyclic spanning subgraph with $n-j$ components has $j$ edges, grouping the terms in Lemma~\ref{lem: PIE} gives
    \[
    P_{DP}(G,\mathcal{H})
    =
    \sum_{j=0}^{n-1}(-1)^ja_{n-j}q^{n-j}.
    \]
\end{proof}

\section{Bipartite Graphs}\label{bipartite}
In this section we begin by proving Theorem~\ref{thm: bipartite}.  To prove Theorem~\ref{thm: bipartite} for a bipartite graph $G$ we show that for sufficiently large $q$, if $\mathcal{H}$ is a cycle-shattering cover of $G$, then $P_{DP}(G,\mathcal{H})=P_{DP}(G,q)$.

\begin{customthm} {\bf \ref{thm: bipartite}}
    Suppose $G$ is a bipartite graph with $n$ vertices and $m$ edges . If $q \geq m + 2^m$ and $a_i$ is the number of acyclic spanning subgraphs of $G$ with $i$ components, then 
    \[
        P_{DP}(G,q) = \sum_{j=0}^{n-1}(-1)^ja_{n-j}q^{n-j}.
    \]
\end{customthm}
\begin{proof}
    Suppose $q \geq m + 2^m$. By Lemma~\ref{lem: shattering_existence}, there exists a cycle-shattering, full $q$-fold cover $\mathcal H_S=(L_S,H_S)$ of $G$. We claim that $P_{DP}(G,q) = P_{DP}(G, \mathcal{H}_S)$. It suffices to prove that $P_{DP}(G, \mathcal{H}) - P_{DP}(G, \mathcal{H}_S) \geq 0$ for an arbitrary full $q$-fold cover $\mathcal{H} = (L,H)$ of $G$.  By Lemma~\ref{lem: PIE}, $P_{DP}(G, \mathcal{H}) - P_{DP}(G, \mathcal{H}_S) = \Delta\mathcal{H}$, where $\Delta\mathcal{H}$ denotes the sum 
    \begin{equation}\label{eq:DP-difference}
         \sum_{A \subseteq E(G)} (-1)^{|A|}\big( |\mathcal{F}_{\angles{\mathcal{H}}{A}}| - |\mathcal{F}_{\angles{\mathcal{H}_S}{A}}|\big).
    \end{equation}
    Let $\mathcal{B} = \{A \subseteq E(G) : \angles{G}{A} \text{ contains a cycle}\}$. If $A \notin \mathcal{B}$, then $\angles{G}{A}$ is acyclic, and  Lemmas~\ref{lem:toolbox}(\ref{tool:canon},\ref{tool:product}) and~\ref{lem: acyclic} imply $|\mathcal F_{\angles{\mathcal H}{A}}|=q^{\operatorname{comp}(\angles{G}{A})} =|\mathcal F_{\angles{\mathcal H_S}{A}}|$. Thus the term corresponding to $A$ in~\eqref{eq:DP-difference} is zero. If $A\in\mathcal B$, then $\angles{G}{A}$ contains a cycle and $|\mathcal F_{\angles{\mathcal{H}_S}{A}}|=0$. So,
    \begin{align*}
     \Delta\mathcal{H} 
     &= \sum_{A\in\mathcal{B}}(-1)^{|A|}\left(|\mathcal {F}_{\angles{\mathcal{H}}{A}}|-|\mathcal F_{\angles{\mathcal{H}_S}{A}}|\right)\\
     &= \sum_{A\in\mathcal{B}}(-1)^{|A|}|\mathcal {F}_{\angles{\mathcal{H}}{A}}|
    \end{align*}
    Clearly, if $|\mathcal F_{\angles{\mathcal H}{A}}|=0$ for every $A \in \mathcal{B}$, then  $\Delta\mathcal{H}=0$ and we are done. Hence we may assume there exists an $A\in\mathcal{B}$ and a cycle $C\subseteq\angles{G}{A}$ where $|\mathcal F_{\angles{\mathcal{H}}{A}}|\geq1$. Since $E(C)\subseteq A$, Lemma~\ref{lem:toolbox}(\ref{tool:edgesubset_single}) gives $|\mathcal F_{\angles{\mathcal H}{E(C)}}|\geq|\mathcal F_{\angles{\mathcal H}{A}}|\geq1$. Let $\ell$ be the minimum length among all such cycles. Since $G$ is bipartite, $\ell = 2k$ for some $k \geq 2$.  

    Let $\mathcal{A} = \{A \in \mathcal{B} : |\mathcal{F}_{\angles{\mathcal{H}}{A}}|\geq 1\}$. Then~\eqref{eq:DP-difference} reduces to $ \Delta\mathcal{H} = \sum_{A \in \mathcal{A}} (-1)^{|A|} |\mathcal{F}_{\angles{\mathcal{H}}{A}}|$.  We claim that if $A \in \mathcal{A}$, then every cycle in $\angles{G}{A}$ has length at least $2k$. Suppose to the contrary that there exists a cycle $C' \subseteq \angles{G}{A}$ with $|E(C')|< 2k$. Since $E(C') \subseteq A$, Lemma~\ref{lem:toolbox}(\ref{tool:edgesubset_single}) gives $|\mathcal{F}_{\angles{\mathcal{H}}{A}}| \leq |\mathcal{F}_{\angles{\mathcal{H}}{E(C')}}|$. However, the minimality of $2k$ implies $|\mathcal{F}_{\angles{\mathcal{H}}{E(C')}}|=0$ and therefore $|\mathcal{F}_{\angles{\mathcal{H}}{A}}|=0$, contradicting $A \in \mathcal{A}$.
    
    Thus for every $A \in \mathcal{A}$, every cycle in $\angles{G}{A}$ has length at least $2k$. Therefore $1 \leq \operatorname{comp}({\angles{G}{A}}) \leq n-2k+1$ since there exists at least one component containing at least $2k$ vertices. We now partition $\mathcal{A}$ according to the value of $n - \operatorname{comp}(\angles{G}{A})$. For each $i\in [2k-1:n-1]$, let $\mathcal A_i = \{ A \in \mathcal{A}: \operatorname{comp}(\angles{G}{A})=n-i \}$ and 
    $\Delta S_i = \sum_{A\in\mathcal A_i} (-1)^{|A|} |\mathcal F_{\angles{\mathcal H}{A}}|$.  
    Then we have
    \[ 
    \Delta \mathcal{H} = \sum_{i=2k-1}^{n-1} \Delta S_i. 
    \]
    To finish our proof that $\Delta \mathcal{H} \geq 0$, we determine $\Delta S_{2k-1}$ exactly, give a careful lower bound on $\Delta S_{2k}$ (when $n > 2k$), and give a rough lower bound on any remaining terms.
    
    Consider $\Delta S_{2k-1}$. Let $A\in \mathcal A_{2k-1}$. Then $n-\operatorname{comp}(\angles{G}{A})=2k-1$.  Since $A \in \mathcal{A}$, $\angles{G}{A}$ contains a cycle and every cycle in $\angles{G}{A}$ has length at least $2k$. Thus some component $W$ of $\angles{G}{A}$ contains a cycle and at least $2k$ vertices. Since $n-\operatorname{comp}(\angles{G}{A})=2k-1$, this implies $|V(W)| = 2k$ and any other components of $\angles{G}{A}$ are isolated vertices.  Thus $W$ contains a cycle of length exactly $2k$.   Moreover, $W$ has no additional edges, since any chord of this $2k$-cycle would create a shorter cycle in $\angles{G}{A}$. Therefore $A$ is the edge set of a $2k$-cycle. Write $\rho_{2k}(G,\mathcal H)=\sum_{A\in\mathcal{A}_{2k-1}}|\mathcal F_{\angles{\mathcal H}{A}[2]}|$ and note $\rho_{2k}(G,\mathcal H)>0$. By Lemma~\ref{lem:toolbox}(\ref{tool:product}),
    \begin{align*}
        \Delta S_{2k-1}
        &=
        \sum_{A\in\mathcal A_{2k-1}}
        (-1)^{|A|}
        |\mathcal F_{\angles{\mathcal H}{A}}| \\
        &=
        \sum_{A\in\mathcal A_{2k-1}}
        |\mathcal F_{\angles{\mathcal H}{A}}| \\
        &=
        \sum_{A\in\mathcal A_{2k-1}}
        |\mathcal F_{\angles{\mathcal H}{A}[2]}|q^{n-2k} \\
        &=
        q^{n-2k}\rho_{2k}(G,\mathcal H).
    \end{align*}

    Now, suppose $n > 2k$, and consider $\Delta S_{2k}$.  Let $\mathcal{A}_{2k}^O$ and $\mathcal{A}_{2k}^E$ consist of the elements of $\mathcal{A}_{2k}$ with odd size and even size respectively.  Then,
     \begin{align*}
     \sum_{A \in \mathcal{A}_{2k}} (-1)^{|A|} |\mathcal{F}_{\angles{\mathcal{H}}{A}}|  
     &= \sum_{A \in \mathcal{A}_{2k}^E} (-1)^{|A|} |\mathcal{F}_{\angles{\mathcal{H}}{A}}| + \sum_{A \in \mathcal{A}_{2k}^O} (-1)^{|A|} |\mathcal{F}_{\angles{\mathcal{H}}{A}}|\\
     &\geq \sum_{A \in \mathcal{A}_{2k}^O} (-1)^{|A|} |\mathcal{F}_{\angles{\mathcal{H}}{A}}|.
    \end{align*}    
    Now, we will bound $\sum_{A \in \mathcal{A}_{2k}^O} (-1)^{|A|} |\mathcal{F}_{\angles{\mathcal{H}}{A}}|$ from below. 
    
    Suppose $A\in\mathcal A_{2k}$. Then $n-\operatorname{comp}(\angles{G}{A})=2k$. Since $A \in \mathcal{A}$, every cycle in $\angles{G}{A}$ has length at least $2k$ and $\angles{G}{A}$ contains a cycle. This implies exactly one component of $\angles{G}{A}$ contains a cycle. Indeed, if two components contained cycles, then each would have at least $2k$ vertices, so $\operatorname{comp}(\angles{G}{A}) \leq n- 2(2k-1)$, meaning $n-\operatorname{comp}(\angles{G}{A}) \geq 4k-2 > 2k$, and contradicting $n-\operatorname{comp}(\angles{G}{A})=2k$.

    Let $W$ be the unique component of $\angles{G}{A}$ containing a cycle. Since every cycle in $\angles{G}{A}$ has length at least $2k$, we have $|V(W)|\geq 2k$. Also, $n-2k = \operatorname{comp}(\angles{G}{A}) \leq n-|V(W)|+1$, which implies $|V(W)| \leq 2k+1$.  Hence $|V(W)| \in \{2k, 2k+1\}$.
    
    If $|V(W)|=2k$, then $W$ is a $2k$-cycle. Let $R ={\angles{G}{A}}- V(W)$. Then $R$ is acyclic, $|V(R)| =n-2k$, and $\operatorname{comp}(R) = n-2k-1$. Hence $R$ consists of one $K_2$ and any other components of $\angles{G}{A}$ are isolated vertices.  Therefore, in the case $|V(W)|=2k$, the nontrivial components of $\angles{G}{A}$ are exactly one $C_{2k}$ and one $K_2$. This means $|A| = 2k+1$ and $A \in \mathcal{A}^O_{2k}$.
    
    Now suppose $|V(W)|=2k+1$. If $\angles{G}{A} \neq W$, let $R ={\angles{G}{A}}-V(W)$. Then $R$ is acyclic, $|V(R)| =n-(2k+1)$, and $\operatorname{comp}(R) = n-2k-1$.  Hence $R$ is edgeless, and $\angles{G}{A}$ has only $W$ as a nontrivial component.  Since $G$ is bipartite, $W$ contains no cycle of length $2k+1$. Thus $W$ contains a cycle $C$ of length $2k$ together with one additional vertex $x$. The cycle $C$ has no chord, since any chord would create a cycle in $\angles{G}{A}$ of length less than $2k$.  We will now determine $W$ when $\deg_W(x)=1$ and when $\deg_W(x) \geq 2$. 

    Suppose $\deg_W(x)=1$.  Then $W$ is obtained from $C$ by adding the vertex $x$ and one edge with exactly one endpoint in $V(C)$. This means $|A| = 2k+1$ and $A \in \mathcal{A}^O_{2k}$.
    
    Now, suppose $\deg_W(x)\geq 2$ and $u,v$ are distinct neighbors of $x$ in $V(C)$. Since $u$ and $v$ are both adjacent to $x$, they lie in the same part of the bipartition of $W$. Hence the two $u,v$-paths contained in $C$ have even length. Suppose their lengths are $2r$ and $2k-2r$ where $r\in [k-1]$. So, $W$ contains cycles of lengths $2r+2$ and $2k-2r+2$.  Since every cycle in $\angles{G}{A}$ has length at least $2k$, both of these cycle lengths are at least $2k$. Hence  
    \[
    2r+2\geq 2k
    \qquad\text{and}\qquad
    2k-2r+2\geq 2k,
    \]
    which means $r\geq k-1$ and $r\leq 1$. Since $1\leq r\leq k-1$, it follows that $k=2$. Thus, it is only possible that $\deg_W(x) \geq 2$ when $C$ is a $4$-cycle. Moreover, when there is such an $x$, $\deg_W(x) = 2$ and $W=K_{2,3}$. This means $|A| = 6$ and $A \in \mathcal{A}^E_{2k}$.

    Thus, if $A\in\mathcal{A}_{2k}^O$, then there is a $2k$-cycle $C$ and edge $e\in E(G) - E(C)$ such that $A=E(C)\cup \{e\}$, where $e$ has at most one endpoint in $V(C)$. Since $E(C) \subseteq A$ and $A \in \mathcal{A}_{2k}$, by Lemma~\ref{lem:toolbox}(\ref{tool:edgesubset_single}), $1\leq |\mathcal{F}_{\angles{\mathcal{H}}{A}}| \leq |\mathcal{F}_{\angles{\mathcal{H}}{E(C)}}|$; so $E(C) \in \mathcal{A}_{2k-1}$.  Now, we will establish a bijection between $\mathcal{A}^O_{2k}$ and a set related to $\mathcal{A}_{2k-1}$ in order to obtain a useful formula for $(-1)^{|A|} |\mathcal{F}_{\angles{\mathcal{H}}{A}}|$ whenever $A \in \mathcal{A}_{2k}^O$. 
    
    For each $A_0 \in \mathcal{A}_{2k-1}$, let $\mathcal{E}(A_0)= \{xy\in E(G) - A_0: |\{x,y\} \cap U(A_0)|\leq 1\}$. Let $\phi: \mathcal{A}_{2k}^O \to \bigcup_{A_0 \in \mathcal{A}_{2k-1}} (\{A_0\} \times \mathcal{E}(A_0))$ be the function that maps each $A \in \mathcal{A}^{O}_{2k}$ to $(E(C_A),e_A)$ where $C_A$ is the $2k$-cycle in $\angles{G}{A}$ and $e_A$ is the edge in $E(G) - E(C_A)$ satisfying $A=E(C_A)\cup \{e_A\}$.  
    By the preceding paragraph $\phi$ is a function and it is easy to prove it is also a bijection.

    Suppose $A \in \mathcal{A}_{2k}^O$ and $\phi(A) = (E(C_A),e_A)$.  Since $A = E(C_A) \cup \{e_A\}$, we have
    \[
        (-1)^{|A|} |\mathcal{F}_{\angles{\mathcal{H}}{A}}|=-|\mathcal{F}_{\angles{\mathcal{H}}{E(C_A)\cup\{e_A\}}}|.
    \]
    
    We claim for each $A \in \mathcal{A}_{2k}^O$ that $|\mathcal{F}_{\angles{\mathcal{H}}{{E(C_A)\cup\{e_A\}}}}| \leq q^{n-2k-1}|\mathcal{F}_{\angles{\mathcal{H}}{E(C_A)}[2]}|$. To prove this, fix $B \in \mathcal{A}_{2k}^O$. Write $e_B =xy$.  First suppose $\{x,y\} \cap V(C_B) = \varnothing$. Then $\angles{G}{B}[2]$ has as its components $C_B$ and the $K_2$ whose edge is $e_B$. Also, since $|U(B)| = 2k+2$, Lemmas~\ref{lem:toolbox}(\ref{tool:canon},\ref{tool:product}) and~\ref{lem: acyclic} give 
    \[
        |\mathcal{F}_{\angles{\mathcal{H}}{{E(C_B)\cup\{e_B\}}}}|
        =
        q^{n-(2k+2)}
        |\mathcal{F}_{\angles{\mathcal{H}}{E(C_B)}[2]}|\cdot|\mathcal{F}_{\angles{\mathcal{H}}{\{e_B\}}[2]}|
        =
        q^{n-2k-1}
        |\mathcal F_{\angles{\mathcal H}{E(C_B)}[2]}|.
    \]
    
    Next, suppose without loss of generality that $x\in V(C_B)$ and $y\notin V(C_B)$. Then $\angles{G}{B}[2]$ is connected and consists of the $2k$-cycle $C_B$ and the degree one vertex $y$ with unique neighbor $x$. Since $E(C_B)\subseteq B$ and $\angles{G}{B}[2]$ is connected, Lemma~\ref{lem:toolbox}(\ref{tool:edgesubset_double}) gives $|\mathcal{F}_{\angles{\mathcal{H}}{E(C_B)\cup\{e_B\}}[2]}|\leq |\mathcal F_{\angles{\mathcal{H}}{E(C_B)}[2]}|$. Also, since $|U(B)| = 2k+1$, Lemma~\ref{lem:toolbox}(\ref{tool:canon},\ref{tool:product}) gives
    \[
        |\mathcal{F}_{\angles{\mathcal{H}}{{E(C_B)\cup\{e_B\}}}}|
        =
        q^{n-(2k+1)}
        |\mathcal{F}_{\angles{\mathcal{H}}{E(C_B)\cup\{e_B\}}[2]}|
        \leq 
        q^{n-2k-1}
        |\mathcal F_{\angles{\mathcal H}{E(C_B)}[2]}|.
    \]

    Consequently, for each $A \in \mathcal{A}_{2k}^O$, $|\mathcal{F}_{\angles{\mathcal{H}}{{E(C_A)\cup\{e_A\}}}}| \leq q^{n-2k-1}|\mathcal{F}_{\angles{\mathcal{H}}{E(C_A)}[2]}|$. Also, since $|\mathcal{E}(A_0)| \leq m-2k$ for each $A_0 \in \mathcal{A}_{2k-1}$, we may write
    \begin{align*}
    \sum_{A\in\mathcal{A}_{2k}^O}(-1)^{|A|}
    |\mathcal{F}_{\angles{\mathcal{H}}{A}}|
    &=-\sum_{A_0\in\mathcal{A}_{2k-1}}\sum_{e\in \mathcal{E}(A_0)}|\mathcal{F}_{\angles{\mathcal{H}}{A_0\cup\{e\}}}|\\
    &\geq -\sum_{A_0\in\mathcal{A}_{2k-1}}\sum_{e\in \mathcal{E}(A_0)} q^{n-2k-1}
    |\mathcal F_{\angles{\mathcal H}{A_0}[2]}|\\
    &\geq -(m-2k)\sum_{A_0\in\mathcal{A}_{2k-1}} q^{n-2k-1}
    |\mathcal F_{\angles{\mathcal H}{A_0}[2]}|\\
    &= -(m-2k)q^{n-2k-1} \rho_{2k}(G,\mathcal H),
    \end{align*}
    and it follows that $\Delta S_{2k} \geq-(m-2k)\rho_{2k}(G,\mathcal H)q^{n-2k-1}$.

    Finally, if $n > 2k+1$, we bound the terms $\Delta S_i$ for each $i\in [2k+1: n-1]$.  Notice that for each $i \in [2k+1:n-1]$ and $A \in \mathcal{A}_i$, Lemma~\ref{lem:toolbox}(\ref{tool:canon},\ref{tool:product}) implies $|\mathcal{F}_{\angles{\mathcal{H}}{A}}|\leq q^{\operatorname{comp}(\angles{G}{A})} \leq q^{n-2k-1}$.   Since the total number of subsets of $E(G)$ is $2^m$, it follows that 
    \[ 
        \sum_{i=2k+1}^{n-1}\Delta S_i \geq -2^m q^{n-2k-1}. 
    \]

    Bringing all this together, we obtain
    \begin{align*}
        P_{DP}(G,\mathcal H)-P_{DP}(G,\mathcal H_S)
        &=
        \sum_{i=2k-1}^{n-1}\Delta S_i \\
        &\geq
        q^{n-2k}\rho_{2k}(G,\mathcal H)
        -(m-2k)q^{n-2k-1}\rho_{2k}(G,\mathcal H)
        -2^m q^{n-2k-1} \\
        &=
        q^{n-2k-1}
        \left(
            \bigl(q-(m-2k)\bigr)\rho_{2k}(G,\mathcal H)
            -2^m
        \right).
    \end{align*}
    One can note that this lower bound holds even when $n \in \{2k,2k+1\}$.  Finally, since $\rho_{2k}(G,\mathcal H)\geq 1$ and $q\geq m + 2^m$,
    \[
        \bigl(q-(m-2k)\bigr)\rho_{2k}(G,\mathcal H)-2^m
        \geq q-(m-2k)-2^m
        \geq 0.
    \]
    Consequently,
    \[
        P_{DP}(G,\mathcal{H})-P_{DP}(G,\mathcal{H}_S)\geq 0.
    \]
    Since $\mathcal{H}$ was an arbitrary $q$-fold full cover, $\mathcal H_S$ minimizes the number of $\mathcal{H}$-colorings among all $q$-fold covers of $G$. Hence $P_{DP}(G,q)=P_{DP}(G,\mathcal H_S)$. Since $\mathcal{H}_S$ is a cycle-shattering cover, Lemma~\ref{lem: shattering_existence} implies  $P_{DP}(G,q) = \sum_{j=0}^{n-1}(-1)^ja_{n-j}q^{n-j}$.
\end{proof}
Thus, the DP color function of a bipartite graph is eventually polynomial.  As noted in the introduction, it would be interesting to determine how low one could take the bound on $q$ in Theorem~\ref{thm: bipartite}.  As a starting point for the study of such an optimal lower bound, we determine an exact formula for $P_{DP}(K_{2,n},q)$ whenever $q \in \N$.  To do that, we first need the following lemma.

\begin{lem}\label{lem:min}
    Let $q,n,k\in\N$ with $q\geq 3$. For any $s\in [0: nk]$, among all  $(a_1, \dots, a_k) \in [0:n]^k$ satisfying $\sum_{\ell=1}^ka_\ell=s$, the sum  $\sum_{\ell=1}^k (q-1)^{a_\ell}(q-2)^{n-a_\ell}$ is minimized if and only if $|a_i-a_j|\leq 1$ for all $i,j\in[k]$.
\end{lem}
\begin{proof}
    Note the result is obvious when $k=1$. Thus we may assume that $k \geq 2$.  We call a tuple $\mathbf{x}=(x_1,\ldots,x_k)\in[0:n]^k$ \emph{balanced} if $|x_i-x_j|\leq 1$ for all $i,j\in[k]$.

    Fix some $s\in [0: nk]$. For the forward direction, we proceed via contradiction. Suppose that $\mathbf{b}=(b_1,\ldots,b_k)$ minimizes the indicated sum among all tuples in $[0:n]^k$ whose coordinates sum to $s$, and suppose that $\mathbf{b}$ is not balanced. Thus there exist $i,j \in [k]$ such that $b_i \geq b_j+2$.  Let $\mathbf{b}' = (b_1',\dots,b_k') \in [0:n]^k$, where
    \[
        b'_\ell=
        \begin{cases}
            b_\ell & \text{ if } \ell \notin \{i,j\} \\
            b_i-1 & \text{ if } \ell = i\\
            b_j+1 & \text{ if } \ell = j.
        \end{cases}
    \]
    Note that the sum of the coordinates of $\textbf{b}'$ is $s$.  To obtain our contradiction, we will show that $\sum_{\ell=1}^k (q-1)^{b_\ell}(q-2)^{n-b_\ell} > \sum_{\ell=1}^k (q-1)^{b'_\ell}(q-2)^{n-b'_\ell}$. 
    \begin{align*}
        &\sum_{\ell=1}^k (q-1)^{b_\ell}(q-2)^{n-b_\ell}
        - \sum_{\ell=1}^k (q-1)^{b'_\ell}(q-2)^{n-b'_\ell} \\
        &= (q-1)^{b_i-1}(q-2)^{n-b_i}
             - (q-1)^{b_j}(q-2)^{n-b_j-1} \\
        &= (q-1)^{b_j}(q-2)^{n-b_i}
            \left((q-1)^{b_i-b_j-1}-(q-2)^{b_i-b_j-1}\right).
    \end{align*}
    Since $b_i-b_j-1\geq 1$ and $q-1>q-2\geq 1$, we have $\sum_{\ell=1}^k (q-1)^{b_\ell}(q-2)^{n-b_\ell} > \sum_{\ell=1}^k (q-1)^{b'_\ell}(q-2)^{n-b'_\ell}$, contradicting the minimality of $\mathbf{b}$. Therefore the elements in $ [0:n]^k$ whose coordinates sum to $s$ and minimize the desired sum are balanced.
    
    Conversely, suppose $\mathbf{b} =(b_1, \dots, b_k) \in [0:n]^k$ is balanced and $\sum_{\ell=1}^k b_\ell=s$.  Let $\mathbf{c} =(c_1, \dots, c_k) \in [0:n]^k$ be a minimizer of $\sum_{\ell=1}^k (q-1)^{a_\ell}(q-2)^{n-a_\ell}$ among all tuples $(a_1, \dots, a_k) \in [0:n]^k$ whose coordinates sum to $s$. By the forward direction, $\mathbf{c}$ is balanced. We claim that $\sum_{\ell=1}^k (q-1)^{b_\ell}(q-2)^{n-b_\ell} = \sum_{\ell=1}^k (q-1)^{c_\ell}(q-2)^{n-c_\ell}$, and thus $\mathbf{b}$ also minimizes the desired sum.
    
    Write $s=kp+r$, where $p\in[0:n]$ and $r \in [0: k-1]$. An easy argument can be used to show that every balanced tuple in $[0:n]^k$ whose coordinates sum to $s$ has exactly $r$ coordinates equal to $p+1$ and $k-r$ coordinates equal to $p$.  
    So, both $\mathbf{b}$ and $\mathbf{c}$ have exactly $r$ coordinates equal to $p+1$ and exactly $k-r$ coordinates equal to $p$.  It follows that
    \[
    \sum_{\ell=1}^k (q-1)^{b_\ell}(q-2)^{n-b_\ell}
    =
    \sum_{\ell=1}^k (q-1)^{c_\ell}(q-2)^{n-c_\ell}, 
    \]
    as desired.
\end{proof}

Now we prove Theorem~\ref{thm:k2n} which we restate.
\begin{customthm} {\bf \ref{thm:k2n}}
Suppose $q, n \in \N$ and that $p$ and $r$ are nonnegative integers satisfying $qn=q^{2}p+r$ where $p = \lfloor n/q \rfloor$. Then $P_{DP}(K_{2,n},q) = \sum_{i=1}^{q^{2}}(q-1)^{a_{i}}(q-2)^{n-a_{i}}$ where, for each $j \in [q^2]$,
\[
    a_j=
    \begin{cases}
    p+1 & \text{ if } j \leq r \\
    p & \text{ if } j > r.
    \end{cases}
\]
\end{customthm}
\newcommand{\KTwoFourFigureWidth}{0.8\linewidth}
\begin{figure}[!t]
\centering
\resizebox{\KTwoFourFigureWidth}{!}{%
\begin{tikzpicture}[
    fiber/.style={
        draw=black!45,
        fill=gray!18,
        ellipse,
        minimum width=1.16cm,
        minimum height=1.90cm,
        inner sep=0pt,
        line width=0.45pt
    },
    vtx/.style={
        circle,
        draw=black!70,
        fill=white,
        minimum size=4.8mm,
        inner sep=0pt,
        font=\scriptsize
    },
    ed/.style={
        draw=black!42,
        line width=0.34pt,
        line cap=round
    },
    bolded/.style={
        draw=black,
        line width=1.15pt,
        line cap=round
    },
    dasheded/.style={
        draw=black,
        line width=1.00pt,
        dash pattern=on 3.2pt off 2.0pt,
        line cap=round,
        preaction={
            draw=white,
            line width=1.65pt,
            line cap=round
        }
    },
    tablebold/.style={
        draw=black,
        line width=0.95pt
    },
    tabledash/.style={
        draw=black,
        line width=0.95pt,
        dash pattern=on 2.5pt off 1.6pt,
        preaction={
            draw=white,
            line width=1.40pt
        }
    },
    lab/.style={font=\small},
    matcell/.style={font=\small},
    every node/.style={black}
]

\foreach \name/\x/\y/\label in {
    xone/0/0/{$L^*(x_1)$},
    xtwo/9.2/0/{$L^*(x_2)$}
} {
    \node[fiber] (\name-fiber) at (\x,\y) {};
    \node[lab, anchor=south, yshift=-1.5pt]
        at (\name-fiber.north) {\label};
    \foreach \g/\dy in {0/0.52,1/0,2/-0.52} {
        \node[vtx] (\name\g) at (\x,\y+\dy) {$\g$};
    }
}

\foreach \name/\y/\label in {
    yone/3.42/{$L^*(y_1)$},
    ytwo/1.14/{$L^*(y_2)$},
    ythree/-1.14/{$L^*(y_3)$},
    yfour/-3.42/{$L^*(y_4)$}
} {
    \node[fiber] (\name-fiber) at (4.6,\y) {};
    \node[lab, anchor=south, yshift=-1.5pt]
        at (\name-fiber.north) {\label};
    \foreach \g/\dy in {0/0.52,1/0,2/-0.52} {
        \node[vtx] (\name\g) at (4.6,\y+\dy) {$\g$};
    }
}

\foreach \Y in {yone,ytwo,ythree,yfour} {
    \foreach \g in {0,1,2} {
        \draw[ed] (xone\g) -- (\Y\g);
    }
}

\foreach \g in {0,1,2} {
    \draw[ed] (yone\g) -- (xtwo\g);
}

\draw[ed] (ytwo0) -- (xtwo2);
\draw[ed] (ytwo1) -- (xtwo0);
\draw[ed] (ytwo2) -- (xtwo1);

\draw[ed] (ythree0) -- (xtwo1);
\draw[ed] (ythree1) -- (xtwo2);
\draw[ed] (ythree2) -- (xtwo0);

\foreach \g in {0,1,2} {
    \draw[ed] (yfour\g) -- (xtwo\g);
}

\draw[bolded] (xone0) -- (yone0);
\draw[bolded] (yone0) -- (xtwo0);
\draw[bolded] (xone0) -- (yfour0);
\draw[bolded] (yfour0) -- (xtwo0);

\draw[dasheded] (xone1) -- (ythree1);
\draw[dasheded] (ythree1) -- (xtwo2);

\foreach \name/\x/\y in {
    xone/0/0,
    xtwo/9.2/0,
    yone/4.6/3.42,
    ytwo/4.6/1.14,
    ythree/4.6/-1.14,
    yfour/4.6/-3.42
} {
    \foreach \g/\dy in {0/0.52,1/0,2/-0.52} {
        \node[vtx] at (\x,\y+\dy) {$\g$};
    }
}

\begin{scope}[shift={(11.5,-1.08)}]
    \draw[line width=0.4pt] (0,0) rectangle (3.6,2.16);

    \foreach \x in {1.2,2.4} {
        \draw[line width=0.4pt] (\x,0) -- (\x,2.16);
    }

    \foreach \y in {0.72,1.44} {
        \draw[line width=0.4pt] (0,\y) -- (3.6,\y);
    }

    \draw[tablebold] (0,1.44) rectangle (1.2,2.16);
    \draw[tabledash] (2.4,0.72) rectangle (3.6,1.44);

    \node at (0.6,2.33) {$\beta=0$};
    \node at (1.8,2.33) {$\beta=1$};
    \node at (3.0,2.33) {$\beta=2$};

    \node[anchor=east] at (-0.15,1.80) {$\alpha=0$};
    \node[anchor=east] at (-0.15,1.08) {$\alpha=1$};
    \node[anchor=east] at (-0.15,0.36) {$\alpha=2$};

    \node[matcell] at (0.6,1.80) {$2$};
    \node[matcell] at (1.8,1.80) {$1$};
    \node[matcell] at (3.0,1.80) {$1$};

    \node[matcell] at (0.6,1.08) {$1$};
    \node[matcell] at (1.8,1.08) {$2$};
    \node[matcell] at (3.0,1.08) {$1$};

    \node[matcell] at (0.6,0.36) {$1$};
    \node[matcell] at (1.8,0.36) {$1$};
    \node[matcell] at (3.0,0.36) {$2$};
\end{scope}
\end{tikzpicture}%
}
\caption{\small
The left part of the figure depicts the $3$-fold cover $\mathcal{H}^*$ of $K=K_{2,4}$ used in the proof of Theorem~\ref{thm:k2n}. For each $v\in V(K)$, a gray oval containing three vertices represents $L^*(v)=\{(v,\gamma):\gamma\in\mathbb{Z}_3\}$.  The vertices in each list are labeled by their second coordinates.  The bold solid paths illustrate that $1,4\in B^*_{(0,0)}$, while the bold dashed path illustrates that $3\in B^*_{(1,2)}$. The right table records the values $b^*_{(\alpha,\beta)}$, with rows indexed by $\alpha$ and columns indexed by $\beta$.
The solid box marks $b^*_{(0,0)}$, and the dashed box marks $b^*_{(1,2)}$.
Consequently,
$P_{DP}(K,\mathcal{H}^*)
=\sum_{(\alpha,\beta)\in\mathbb{Z}_3^2}
(3-1)^{b^*_{(\alpha,\beta)}}
(3-2)^{4-b^*_{(\alpha,\beta)}}
=3\cdot2^2+6\cdot2=24$.
}
\label{fig:K24}
\end{figure}
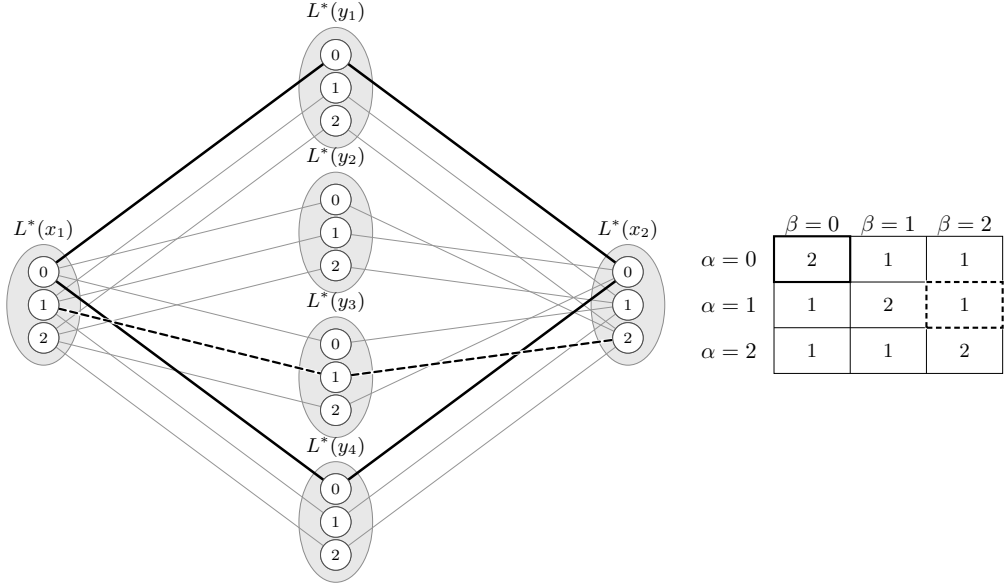
\begin{proof}
    Let $K=K_{2,n}$ with bipartition $X=\{x_h:h\in[2]\}$ and $Y=\{y_t:t\in[n]\}$. If $q=1$, then $P_{DP}(K,1)=0$ which is the same as the stated formula. 
    
    Now suppose  $q=2$. When $n=1$, $K$ is a path on three vertices which means $P_{DP}(K,2)=2$.  The stated formula gives $2$ since $p=0$ and $r=2$. When $n\geq 2$, $K$ contains a $4$-cycle, so $\chi_{DP}(K) \geq 3$ and  $P_{DP}(K,2)=0$.  The stated formula also gives $0$.

    Thus we may assume $q\geq 3$. Set $\sigma=\sum_{j=1}^{q^2}(q-1)^{a_j}(q-2)^{n-a_j}$ as defined in the statement of the theorem. We first prove that $P_{DP}(K,q)\geq\sigma$.  Let $\mathcal{H} = (L,H)$ be a full $q$-fold cover of $K$ such that $P_{DP}(K,\mathcal{H}) = P_{DP}(K,q)$. For each $v \in V(K)$, write $L(v) = \{(v,\gamma) : \gamma \in [q]\}$.  Since $\mathcal{H}$ is full and $x_hy_t\in E(K)$ for each $h\in[2]$ and $t\in[n]$, $E_H(L(x_h), L(y_t))$ is a perfect matching. For each $h\in[2]$, $t\in[n]$, define the map $\pi_{h,t}: [q] \to [q]$ by letting $\pi_{h,t}(\gamma)$ be the unique element of $[q]$ such that $(x_h,\gamma)(y_t,\pi_{h,t}(\gamma))\in E(H)$.  Since $E_H(L(x_h),L(y_t))$ is a perfect matching, $\pi_{h,t}$ is a bijection and hence a permutation of $[q]$.  
    
    Let $\mathcal{I}_{\mathcal{H}}$ be the set of proper $\mathcal{H}$-colorings of $K$. Then $P_{DP}(K,\mathcal{H})=|\mathcal{I}_{\mathcal{H}}|$. For each $(\alpha, \beta)\in [q]^2$, let 
    $\mathcal{I}_\mathcal{H}(\alpha, \beta) =  \{I\in \mathcal{I}_{\mathcal{H}}: (x_1, \alpha), (x_2, \beta) \in I \}.
    $
    Note that since $q \geq 3$, $\mathcal{I}_{\mathcal{H}}(\alpha, \beta) \neq \emptyset$ for all $(\alpha, \beta) \in [q]^2$.
    Since every $I \in \mathcal{I}_{\mathcal{H}}$ is a transversal, each such $I$ contains exactly one element of $L(x_1)$ and exactly one element of $L(x_2)$. Hence $\{\mathcal{I}_{\mathcal{H}}(\alpha, \beta): (\alpha, \beta) \in [q]^2\}$ is a partition of $\mathcal{I}_{\mathcal{H}}$, and  $P_{DP}(K,\mathcal{H}) = \sum_{(\alpha,\beta)\in[q]^2} |\mathcal{I}_{\mathcal{H}}(\alpha,\beta)|$.  We now compute $|\mathcal{I}_{\mathcal{H}}(\alpha,\beta)|$ for each $(\alpha,\beta)\in[q]^2$.

    For each $(\alpha,\beta)\in[q]^2$ and $t \in [n]$, let  $ D_t(\alpha,\beta) = \{\delta\in[q]: (x_1,\alpha)(y_t,\delta), (x_2,\beta)(y_t,\delta)\notin E(H)\} $.  Note $|D_t(\alpha,\beta)|=q-1$ if $\pi_{1,t}(\alpha)=\pi_{2,t}(\beta)$ and $|D_t(\alpha,\beta)|=q-2$ otherwise.  Clearly for all $(\alpha, \beta) \in [q]^2$, $|\mathcal{I}_{\mathcal{H}}(\alpha,\beta)|=\prod_{t=1}^{n}|D_t(\alpha,\beta)|$.

    Let $B_{(\alpha, \beta)}=\{t\in[n]:\pi_{1,t}(\alpha)=\pi_{2,t}(\beta)\}$ and $b_{(\alpha,\beta)} = |B_{(\alpha,\beta)}|$ for all $(\alpha, \beta) \in [q]^2$.  Since $|D_t(\alpha,\beta)|=q-1$ precisely when $\pi_{1,t}(\alpha)=\pi_{2,t}(\beta)$, we have $|D_t(\alpha,\beta)|=q-1$ for exactly $b_{(\alpha, \beta)}$ values of $t\in[n]$ and $|D_t(\alpha,\beta)|=q-2$ for the remaining $n-b_{(\alpha, \beta)}$ values of $t$. Hence $|\mathcal{I}_{\mathcal{H}}(\alpha,\beta)|=(q-1)^{b_{(\alpha,\beta)}}(q-2)^{n-b_{(\alpha, \beta)}}$, and it follows that 
    $
        P_{DP}(K,q)
        =
        \sum_{(\alpha,\beta)\in[q]^2}
        (q-1)^{b_{(\alpha,\beta)}}(q-2)^{n-b_{(\alpha,\beta)}}.
    $

    Finally, we claim that  $\sum_{(\alpha,\beta)\in[q]^2} b_{(\alpha,\beta)}=qn$. To see this, fix $t_0\in [n]$. For each $\alpha \in [q]$, there is one unique $\beta \in [q]$ such that $\pi_{1,t_0}(\alpha) = \pi_{2,t_0}(\beta)$ since  $\pi_{2,t_0}$ is a permutation of $[q]$.  Thus for each $t \in [n]$ there exist exactly $q$ ordered pairs $(\alpha,\beta)$ in $[q]^2$ with the property that $t \in B_{(\alpha,\beta)}$.  So,  $\sum_{(\alpha,\beta)\in[q]^2} b_{(\alpha,\beta)} =qn$.

   The $q^2$ integers $b_{(\alpha,\beta)}$, indexed by $(\alpha,\beta)\in[q]^2$, lie in $[0:n]$ and have sum $qn$. Likewise, the integers $a_i$, indexed by $i\in[q^2]$, lie in $[0:n]$ and satisfy $\sum_{i=1}^{q^2}a_i=q^2p+r=qn$. Moreover, $|a_i-a_j|\leq1$ for all $i,j\in[q^2]$. Therefore, Lemma~\ref{lem:min}, applied with $k=q^2$ and $s=qn$, gives
    \[
        P_{DP}(K,q)
        =
        \sum_{(\alpha,\beta)\in[q]^2}
        (q-1)^{b_{(\alpha,\beta)}}(q-2)^{n-b_{(\alpha,\beta)}}
        \geq
        \sigma.
    \]

    It remains to show that $P_{DP}(K,q)\leq\sigma$. It is sufficient to construct a $q$-fold cover $\mathcal{H}^{*}=(L^{*},H^{*})$ of $K$ such that $P_{DP}(K,\mathcal{H}^{*})=\sigma$. For each $v\in V(K)$, let $L^{*}(v)=\{(v,\gamma):\gamma\in \Z_q\}$.
    In this construction, all arithmetic is done in $\Z_q$, with the exception of calculating $P_{DP}(K, \mathcal{H}^*)$. For each $t \in [n]$, let $\pi^*_{1,t}, \pi^*_{2,t}: \Z_q \to \Z_q$ be the permutations defined by $\pi^*_{1,t}(\gamma) = \gamma$ and $\pi^*_{2,t}(\gamma) = \gamma + t-1$ . For each $h \in [2]$, $t \in [n]$, and $\gamma \in \Z_q$, add the edge $(x_h,\gamma)(y_t,\pi^*_{h,t}(\gamma))$ to $E(H^*)$. Figure~\ref{fig:K24} depicts the cover $\mathcal{H}^*$ of $K$ in the case $q=3$ and $n=4$.  
    
    We now show that $P_{DP}(K,\mathcal H^*)=\sigma$. For each $(\alpha,\beta)\in\Z_q^2$, let $B_{(\alpha, \beta)}^* = \{t\in[n]:\pi^*_{1,t}(\alpha)=\pi^*_{2,t}(\beta)\}$. Note the elements of $B_{(\alpha, \beta)}^*$ are integers and not elements of $\mathbb{Z}_q$. See Figure~\ref{fig:K24} for an illustration of these sets.  Write $b^*_{(\alpha,\beta)} =|B_{(\alpha, \beta)}^*|$. The counting argument above applies to $\mathcal H^*$ with $\Z_q$ in place of $[q]$; therefore $P_{DP}(K,\mathcal H^*)=\sum_{(\alpha,\beta)\in\Z_q^2}(q-1)^{b^*_{(\alpha,\beta)}}(q-2)^{n-b^*_{(\alpha,\beta)}}$.

    It remains to show that $\sum_{(\alpha,\beta)\in\Z_q^2}(q-1)^{b^*_{(\alpha,\beta)}}(q-2)^{n-b^*_{(\alpha,\beta)}} = \sigma$. To do this we determine  $b^*_{(\alpha, \beta)}$ for each $(\alpha, \beta) \in \Z_q^2$.  Since $\pi^*_{1,t}(\alpha)=\alpha$ and $\pi^*_{2,t}(\beta)=\beta+t-1$, $b^*_{(\alpha,\beta)}$ is determined by $\alpha-\beta$ and $n$.  For each $d \in \Z_q$, set $S_d = \{(\alpha, \beta) \in \Z_q^2 : \alpha - \beta = d\}$.  Then $\{S_0,\ldots,S_{q-1}\}$ is a partition of $\Z_q^2$.  Moreover, for each fixed $d\in[0:q-1]$, the map $\beta \mapsto (\beta+ d,\beta)$ is a bijection from $\Z_q$ to $S_d$; so $|S_d|=q$.

    Now fix $d\in \Z_q$ and suppose $(\alpha,\beta)\in S_d$. Then $\alpha-\beta= d$. Hence, for each $t\in[n]$, the condition $\alpha=\beta+t-1$ is equivalent to $t-1 = d$. Thus $B^*_{(\alpha, \beta)} = \{ t \in [n]: t-1 = d\}$.  
    The number of integers $t \in [n]$ with the property that $t-1$ is congruent to $d$ mod $q$ is $p+1$ when $d < (n \mod q)$, otherwise it is
     $p$ when $d \geq (n \mod q)$. So, if $s= n \mod q$, then
    \[
        b^*_{(\alpha, \beta)}=
        \begin{cases}
             p+1 & \text{if } d \in [0:s-1]\\
            p & \text{if } d \in [s: q-1]
           
        \end{cases}.
    \]
    The table in Figure~\ref{fig:K24} organizes the values of $ b^*_{(\alpha, \beta)}$ for the case $q=3$ and $n=4$.

    For the remainder of the proof we use standard arithmetic. Notice $sq =r$. Since $|S_d|=q$ for every $d\in[0:q-1]$, the number of ordered pairs $(\alpha,\beta)\in\Z_q^2$ satisfying $b^*_{(\alpha,\beta)}=p+1$ is $\sum_{d=0}^{s-1}|S_d|=sq$. Similarly, the number of ordered pairs $(\alpha,\beta)\in\Z_q^2$ satisfying $b^*_{(\alpha,\beta)}=p$ is $\sum_{d=s}^{q-1}|S_d|=(q-s)q=q^2-sq$.  Thus the $q^2$ values $b^*_{(\alpha,\beta)}$, indexed by $(\alpha,\beta)\in\Z_q^2$, consist of exactly $r$ copies of $p+1$ and $q^2-r$ copies of $p$. Therefore
    \[
        P_{DP}(K,\mathcal H^*)
        =
        \sum_{(\alpha,\beta)\in\Z_q^2}
        (q-1)^{b^*_{(\alpha,\beta)}}
        (q-2)^{n-b^*_{(\alpha,\beta)}}
        =
        \sum_{j=1}^{q^2}
        (q-1)^{a_j}(q-2)^{n-a_j}
        =
        \sigma.
    \] 
\end{proof} 
 One can note from our result that $P_{DP}(K_{2,n},q)$ is identical to the formula in Theorem~\ref{thm: bipartite} whenever $q \geq n$.  Moreover, when $q \geq n$, the cover $\mathcal{H}^*$ constructed in the proof of Theorem~\ref{thm:k2n} is cycle-shattering.

\section{Graphs with Girth that is Even}\label{girth}
In this section we prove Theorem~\ref{thm:evencycles}. We first recall the following fact about chromatic polynomials. 
\begin{pro}[\cite{DKT05}] \label{prop: whitney}
     Let $G$ be an $n$-vertex graph with $m$ edges and girth $\ell$ and write \\ $P(G,q) = \sum_{i=0}^{n} (-1)^ia_{i}q^{n-i}$.  Then $a_i = \binom{m}{i}$ for $i \in [0:\ell-2]$, and  $a_{\ell-1} = \binom{m}{\ell-1} - t$ where  $t$ is the number of $\ell$-cycles in $G$.
\end{pro}
We now prove Theorem~\ref{thm:evencycles}, which we restate.
\begin{customthm} {\bf \ref{thm:evencycles}}
     Let $g \geq 3$ be odd. Suppose $G$ is an $n$-vertex graph with $m$ edges and girth $g+1$, and let $t$ be the number of $(g+1)$-cycles in $G$. Then 
     \[
        P(G,q) - P_{DP}(G,q) = t q^{n-g} + O(q^{n-g-1})
     \]
     as $q \to \infty$.  
\end{customthm}
\begin{proof}
    By Proposition~\ref{prop: whitney}, there exist fixed integers $c_{g+1},\dots,c_n$ such that
    \[
        P(G,q)
        =
        \sum_{i=0}^{g-1}(-1)^i\binom{m}{i}q^{n-i} 
        + (-1)^g\left(\binom{m}{g} 
        - t\right)q^{n-g}
        + \sum_{i=g+1}^{n}(-1)^i c_iq^{n-i}.
    \]
    Set $R(q) = \sum_{i=g+1}^{n}(-1)^i c_iq^{n-i}$. Then $R(q) = O(q^{n-g-1})$. Also, 
    \[
        P(G,q)
        =
        \sum_{i=0}^{g}(-1)^i\binom{m}{i}q^{n-i}
        +
        tq^{n-g}
        + R(q).
    \]
    We first prove  $P(G,q) - P_{DP}(G,q) \leq   tq^{n-g}+O(q^{n-g-1})$ by showing  $P_{DP}(G,q) \geq  \sum_{i=0}^{g}(-1)^i\binom{m}{i}q^{n-i}
    -2^m q^{n-g-1}$ for every $q \in \N$. Fix $q\in\mathbb N$ and let $\mathcal{H}^*=(L^*,H^*)$ be a full $q$-fold cover of $G$ such that $P_{DP}(G,\mathcal{H}^*)=P_{DP}(G,q)$. For any $q$-fold cover $\mathcal{H}$ of $G$ and  $r \in [0: m]$, let \\ $S_r(\mathcal{H}) = (-1)^r\sum_{A\subseteq E(G), |A|=r}|\mathcal{F}_{\angles{\mathcal{H}}{A}}|$. By Lemma~\ref{lem: PIE},
    \begin{equation}~\label{eqn:PIE}
                P_{DP}(G,\mathcal{H}^*)= \sum_{r=0}^{m}S_r(\mathcal{H}^*)= \sum_{r=0}^{g}S_r(\mathcal{H}^*)+S_{g+1}(\mathcal{H}^*)+\sum_{r=g+2}^{m}S_r(\mathcal{H}^*).
    \end{equation}
    We proceed by finding $\sum_{r=0}^{g}S_r(\mathcal{H}^*)$ exactly and bound both $S_{g+1}(\mathcal{H}^*)$ and $\sum_{r=g+2}^{m}S_r(\mathcal{H}^*)$ from below.

    First consider $\sum_{r=0}^{g}S_r(\mathcal{H}^*)$. Fix $r\in[0:g]$. Since $G$ has girth $g+1$, for every $A \subseteq E(G)$ such that $|A|=r$, $\angles{G}{A}$ is acyclic. Hence $\operatorname{comp}(\angles{G}{A})=n-r$ and Lemmas~\ref{lem:toolbox}(\ref{tool:canon},\ref{tool:product}) and~\ref{lem: acyclic} give
    $
       |\mathcal{F}_{\angles{\mathcal{H}^*}{A}}|=q^{\operatorname{comp}(\angles{G}{A})} = q^{n-r} 
    $.
    Therefore, $S_r(\mathcal{H}^*) = (-1)^r\binom{m}{r}q^{n-r}$ for every $r \in [0:g]$.

    Now consider $S_{g+1}(\mathcal{H}^*)$. Since $g$ is odd,  $S_{g+1}(\mathcal{H}^*)= \sum_{A\subseteq E(G), |A|=g+1}(-1)^{g+1}|\mathcal{F}_{\angles{\mathcal{H}^*}{A}}|\geq 0$.

    It remains to bound  $\sum_{r=g+2}^{m}S_r(\mathcal{H}^*)$. If $m \leq g+1$, then $\sum_{r=g+2}^{m}S_r(\mathcal{H}^*) =0$, so suppose $m \geq g+2$.  Fix $r\in[g+2:m]$ and suppose $A \subseteq E(G)$ satisfies $|A| =r$. By Lemma~\ref{lem:toolbox}(\ref{tool:canon},\ref{tool:product}), $|\mathcal{F}_{\angles{\mathcal{H}^*}{A}}| \leq q^{\operatorname{comp}(\angles{G}{A})}$. Thus it suffices to show that $\operatorname{comp}(\angles{G}{A}) \leq n-g-1$.  Write $A = \{e_1, \dots, e_r\}$, and, for each $i \in [r]$, let $A_i = \{e_1, \dots, e_i\}$. For $1\leq k\leq j\leq r$ we have $A_k \subseteq A_j$, so $\operatorname{comp}(\angles{G}{A_j}) \leq \operatorname{comp}(\angles{G}{A_k})$.  
    
    Since $G$ has girth $g+1$, $\angles{G}{A_i}$ is acyclic for every $i \in [g]$. We now consider $\angles{G}{A_{g+1}}$. If $\angles{G}{A_{g+1}}$ is acyclic, then $\operatorname{comp}(\angles{G}{A_{g+1}})=n-(g+1)=n-g-1$ and $\operatorname{comp}(\angles{G}{A}) \leq n-g-1$ follows.  
    
    Suppose instead $\angles{G}{A_{g+1}}$ contains a cycle. Then there exists a $(g+1)$-cycle $C$ in $G$ such that $A_{g+1} = E(C)$, so $\operatorname{comp}(\angles{G}{A_{g+1}})=n-(g+1)+1=n-g$. Also, since $e_{g+2} \notin E(C)$, its endpoints cannot both lie in $V(C)$, as this would imply $e_{g+2}$ is a chord of $C$, creating a cycle of length less than $g+1$.  Hence $e_{g+2}$ joins two distinct components of $\angles{G}{A_{g+1}}$, so $\operatorname{comp}(\angles{G}{A_{g+2}})=n-g-1$ and $\operatorname{comp}(\angles{G}{A}) \leq n-g-1$ follows.

    Thus, for every $A\subseteq E(G)$ with $|A|=r$, $|\mathcal{F}_{\angles{\mathcal{H}^*}{A}}| \leq q^{n-g-1}$ and  $S_r(\mathcal{H}^*)=(-1)^r\sum_{A\subseteq E(G),|A|=r}|\mathcal{F}_{\angles{\mathcal{H}^*}{A}}|\geq -\binom{m}{r} q^{n-g-1}$ for every $r\in[g+2:m]$.

    Consequently, using (\ref{eqn:PIE}), we have
    \begin{align*}
        P_{DP}(G, q) &\geq \sum_{r=0}^{g}(-1)^r\binom{m}{r}q^{n-r} + 0 -\sum_{r=g+2}^{m}\binom{m}{r} q^{n-g-1} \\
        &\geq \sum_{r=0}^{g}(-1)^r\binom{m}{r}q^{n-r} -2^m q^{n-g-1}.
    \end{align*}
    Thus $P(G,q) - P_{DP}(G,q) \leq   tq^{n-g}+ R(q) + 2^m q^{n-g-1}$ for every $q \in \N$. Since $R(q)+2^m q^{n-g-1} = O(q^{n-g-1})$ as $q \to \infty$, there is a positive constant $C_1$ such that $P(G,q)-P_{DP}(G,q)
    \le tq^{n-g}+C_1q^{n-g-1}$ provided $q$ is sufficiently large.

    We now prove $P(G,q) - P_{DP}(G,q) \geq   tq^{n-g}+O(q^{n-g-1})$ by showing that $P_{DP}(G,q) \leq \sum_{i=0}^{g}(-1)^i\binom{m}{i}q^{n-i}+ \left(\binom{m}{g+1}-t\right)q^{n-g-1} +2^m q^{n-g-2}$  for every $q \geq 2^m$. Suppose $q \geq 2^m$. By Lemma~\ref{lem: shattering_existence}, there exists a full $q$-fold cover $\mathcal{H}_S = (L_S, H_S)$ of $G$ such that for all $A \subseteq E(G)$, if $\angles{G}{A}$ contains a cycle then $|\mathcal{F}_{\angles{\mathcal{H}_S}{A}}| =0$. By Lemma~\ref{lem: PIE} and the definition of $S_r$ as above, 
    \begin{equation}~\label{eqn:PIE s}
                P_{DP}(G,\mathcal{H}_S)= \sum_{r=0}^{m}S_r(\mathcal{H}_S)= \sum_{r=0}^{g}S_r(\mathcal{H}_S)+S_{g+1}(\mathcal{H}_S)+\sum_{r=g+2}^{m}S_r(\mathcal{H}_S).
    \end{equation}
    We proceed by finding $\sum_{r=0}^{g}S_r(\mathcal{H}_S)$ and $S_{g+1}(\mathcal{H}_S)$ exactly and bound  $\sum_{r=g+2}^{m}S_r(\mathcal{H}_S)$ from above.

    First consider $\sum_{r=0}^gS_r(\mathcal{H}_S)$. As before, for every $r\in[0:g]$ and any $A\subseteq E(G)$ with $|A|=r$, $\angles{G}{A}$ is acyclic. Hence $S_r(\mathcal{H}_S) =(-1)^r\binom mr q^{n-r}$ for every $r\in[0:g]$.

     Now consider $S_{g+1}(\mathcal H_S)$. Since $G$ has girth $g+1$, for any $A\subseteq E(G)$ with $|A|=g+1$,
    $\angles{G}{A}$ contains a cycle if and only if $A$ is the  edge set of a $(g+1)$-cycle. Thus exactly $t$ such sets $A$ are the edge sets of $(g+1)$-cycles , while the remaining $\binom{m}{g+1}-t$ sets $A$ have $\angles{G}{A}$ acyclic.  If $\angles{G}{A}$ contains a cycle, then $|\mathcal{F}_{\angles{\mathcal{H}_S}{A}}|=0$ by Lemma~\ref{lem: shattering_existence}.  Otherwise, $\angles{G}{A}$ is acyclic, so $\operatorname{comp}(\angles{G}{A})=n-g-1$, and Lemmas~\ref{lem:toolbox}(\ref{tool:canon},\ref{tool:product}) and~\ref{lem: acyclic} give $|\mathcal{F}_{\angles{\mathcal{H}_S}{A}}|=q^{n-g-1}$.  It follows that
    \[
        S_{g+1}(\mathcal H_S)
        =
        \left(\binom{m}{g+1}-t\right)q^{n-g-1}.
    \]

    It remains to bound  $\sum_{r=g+2}^{m}S_r(\mathcal{H}_S)$ from above. To do this, we will bound $|S_r(\mathcal H_S)|$ from above for every $r\in[g+2:m]$.  Fix $r\in[g+2:m]$.  For every $A\subseteq E(G)$ with $|A|=r$, if $\angles{G}{A}$ contains a cycle, then $|\mathcal{F}_{\angles{\mathcal{H}_S}{A}}|=0$ by Lemma~\ref{lem: shattering_existence}. Otherwise,  $\angles{G}{A}$ is acyclic, so $\operatorname{comp}(\angles{G}{A})=n-r$, and Lemmas~\ref{lem:toolbox}(\ref{tool:canon},\ref{tool:product}) and~\ref{lem: acyclic} give $|\mathcal{F}_{\angles{\mathcal{H}_S}{A}}| =q^{n-r}$. Therefore, 
    \begin{align*}
        |S_r(\mathcal H_S)|
        &=
        \sum_{\substack{A\subseteq E(G)\\ |A|=r\\
                        \angles{G}{A}\text{ is acyclic}}}
        q^{n-r}
        \leq
        \binom{m}{r}q^{n-r}
        \leq
        \binom{m}{r}q^{n-g-2}.
    \end{align*}
    Thus, summing over $r\in[g+2:m]$ gives $\sum_{r=g+2}^{m}S_r(\mathcal H_S) \leq \sum_{r=g+2}^{m}|S_r(\mathcal H_S)| \leq  2^m q^{n-g-2}$.
    Consequently, using~\eqref{eqn:PIE s}, we have
    \[
    \begin{aligned}
        P_{DP}(G,\mathcal H_S)
        &\leq
        \sum_{r=0}^{g}(-1)^r\binom{m}{r}q^{n-r}+
        \left(\binom{m}{g+1}-t\right)q^{n-g-1}
        +
        2^m q^{n-g-2}.
    \end{aligned}
    \]
    Since $P_{DP}(G,q) \leq P_{DP}(G, \mathcal{H}_S)$, it follows that  $P(G,q) - P_{DP}(G,q) \geq   tq^{n-g}+ R(q) -  (\binom{m}{g+1}-t)q^{n-g-1} -  2^m q^{n-g-2} $.  Since $R(q) - (\binom{m}{g+1}-t)q^{n-g-1} -  2^m q^{n-g-2} = O(q^{n-g-1})$ as $q \to \infty$, there is a positive constant $C_2$ such that $tq^{n-g}-C_2q^{n-g-1}
    \le P(G,q)-P_{DP}(G,q)$ provided $q$ is sufficiently large.  The desired result immediately follows.
\end{proof}
 
{\bf Declaration of generative AI and AI-assisted technologies in the manuscript preparation process.} 
During the preparation of this manuscript, the authors used ChatGPT (OpenAI) between May and July 2026 for limited language editing and for assistance in creating and refining certain figures. All AI-assisted text and graphical material was reviewed, revised, and independently verified by the authors. 

\bibliographystyle{hplain}
\bibliography{bibliography}{}

\end{document}